\documentclass{amsart}
\usepackage{amssymb,latexsym,amscd}
\input xy 
\xyoption{all}
\numberwithin{equation}{section}

\newtheorem{theorem}{Theorem}[section]
\newtheorem{proposition}[theorem]{Proposition}
\newtheorem{conjecture}[theorem]{Conjecture}

\newtheorem{corollary}[theorem]{Corollary}
\newtheorem{remark}[theorem]{Remark}
\newtheorem{lemma}[theorem]{Lemma}
\newtheorem{example}[theorem]{Example}
\newtheorem{definition}[theorem]{Definition}

\newtheorem{claim}[theorem]{Claim}

\newtheorem{acknowledgements}[theorem]{Acknowledgements}

\def\AA{\mathfrak{A}}

\def\ZZ{\mathbb{Z}}
\def\QQ{\mathbb{Q}}
\def\CC{\mathbb{C}}

\def\UU{\mathcal{U}}
\def\II{\mathcal{I}}
\def\JJ{\mathcal{J}}
\def\PP{\mathbb{P}}
\def\1{\mathbb{1}}

\def\gg{\mathfrak{g}}
\def\mm{\mathfrak{m}}
\def\nn{\mathfrak{n}}
\def\hh{\mathfrak{h}}

\def\OO{\mathcal{O}}

\def\PP{\mathbb{P}}

\begin{document}

\title{Cluster Algebras, Symplectic Leaves and Quantum Groups}
\author{Sebastian Zwicknagl}

\begin{abstract}
  This paper investigates the Poisson geometry of cluster algebras and the corresponding ideal theory of quantum cluster algebras. We then show how our approach can be applied to the ring theory of quantized coordinate rings. We give a new construction for the Dixmier map constructed by Yakimov from the space of symplectic leaves on $\CC[G]$ to the space of primitive ideals on $\CC_q[G]$ and give further evidence that this map is a homeomorphism. 
\end{abstract}
 
\maketitle

\tableofcontents
 \section{Introduction}
 The present paper studies the Poisson and non-commutative geometry of cluster algebras and quantum cluster algebras which resemble cluster algebras obtained from Lie theory. It  is part of a program to investigate the ideal theory of quantum groups. It complements \cite{ZW tpc}  and \cite{ZW acy} which focused on the toric Poisson prime ideals and the ideals in cluster algebras arising from acyclic quivers, respectively. 
 
 Many quantized coordinate rings such as the quantum matrices--and, therefore, the quantized coordinate ring $\CC_q[SL_n]$ of $SL_n$--have been proven to admit quantum cluster algebra structures (see Gei\ss, Leclerc and Schr\"oer's \cite{GLSq}), and it was conjectured by Berenstein and Zelevinsky \cite{bz-qclust} that $\CC_q[G]$ is an upper quantum cluster algebra for arbitrary semisimple $G$. Indeed, the semiclassical limit of $\CC_q[G]$  is the so-called standard Poisson structure on $\CC[G]$ which is known to be compatible with the upper cluster algebra structure constructed by Berenstein, Fomin and Zelevinsky in \cite{BFZ}. Cluster algebras were introduced to study the dual canonical bases, however,  there are also a number of prominent questions regarding the ideal theory of quantum groups and their classical limits. A conjecture  that has attracted a lot of attention over the last two decades states that the spectrum of $\CC_q[G]$ and the spectrum of Poisson prime ideals in the semiclassical limit are homeomorphic. This conjecture is an analogue of the Kirillov's Orbit Method  (see e.g.\cite{Kir2}), and it has been proven only in the cases of $SL_2$ and $SL_3$. Even though it has been long known that there exist bijections between these two sets, it is quite non-trivial to give a precise conjectural formulations of this homeomorphism. This was accomplished recently by Yakimov in \cite{Yak-spec} where the reader may also find more background information about the problem. As an application of our results on quantum cluster algebras and their classical limits, we are able to give  some apparently new evidence in this paper that the map is indeed a homeomorphism.     

Recall that a cluster algebra $\AA$ is given by combinatorial data, namely a cluster ${\bf x}=(x_1,\ldots, x_n)\subset A$ and a skew-symmetrizable $m\times n$-integer matrix $B$ where $m<n$. In order to define a quantum cluster algebra $\AA_q$ we also need a skew-symmetric integer matrix $\Lambda$ such that $(B,\Lambda)$ forms a compatible pair (see Section \ref{se:Cluster Algebras} for the definitions). The subalgebra $\CC_\Lambda[x_1,\ldots, x_n]$ generated by the elements of the cluster  ${\bf x}$ in $\AA_q$ forms a quantum affine space (see Appendix \ref{se:Poisson and quantum tori} for the definition) and our main idea is to study the ideal theory of the quantum cluster algebra via its intersection with these quantum affine spaces for various clusters. Indeed, the Laurent phenomenon of \cite{bz-qclust} which asserts that $\AA_q\subset \CC_\Lambda[x_1^\pm 1,\ldots, x_n^\pm 1]$ becomes a key tool in this approach and we will now outline how it allows us to improve on the two known constructions  of the spectra, due to Joseph \cite{Jo1} and Goodearl--Letzter (see e.g. \cite[Sect 2]{brown-goodearl}). 

In both cases, one first studies the torus invariant prime ideals (TIPs) in the algebra $\CC_q[G]$  and then stratifies the prime spectrum by assigning a prime $I$ to a TIP $\II$ if $\II$ is the maximal TIP contained in $I$. In order to consider the stratum assigned to a TIP $\II$ we consider the quotient   $A_\II=\CC_q[G]/\II$ and then localize $A_\II$ at a suitable multiplicative set. In the case of Joseph's approach one chooses the torus invariant elements $N_\II$ of $A_\II$ and in the case of Goodearl-Letzter one would use the torus invariant regular elements $R_\II$. The stratum of $\II$ is homeomorphic to the spectra $spec(A_\II[N_\II^{-1}])\cong  spec(A_\II[R_\II^{-1}])$ which are homeomorphic to the spectrum of a Laurent polynomial ring. Clearly, $N_\II\subset R_\II$.  Now, assume that $R_q[G]$ is a quantum cluster algebra.  Let ${\bf x}$ be a cluster, and denote by ${\bf X_\II}$ the set $\{x_1^{d_1}\ldots x_n^{d_n}: d_j=0\  \text{if}\ x_j\in \II\}$.  If we work with a suitable (for $\II$) cluster ${\bf x}$ then  $N_\II\subset {\bf X}_I\cap A_\II\subset R_\II$, and we can now choose to localize at ${\bf X}_I$.

The main advantage of our approach is that we can use the fact that if $\II\subset \JJ$ are two TIPs and we can find a cluster ${\bf x}$ that is suitable for both $\II$ and $\JJ$, then we can use the fact that  we use localization at related sets (namely cluster monomials from the same cluster) to gain information about the inclusion relations between primes of different strata. Indeed, we can construct inclusion preserving embeddings of the strata into the spectrum of $\CC_\Lambda[x_1,\ldots, x_n]$ (see Corollary \ref{cor:intersection spectra q}). 

As we develop analogous results for the Poisson spectra  of the semi-classical limits, we are able to give some additional evidence that both spaces are homeomorphic. We showcase this approach in the case of $SL_n$ which has a quantum cluster algebra structure.  We establish that for each  pair  of TIPs $\II$ and $\JJ$ with $\II\subset \JJ$ there exists a cluster that is suitable. Moreover, if we restrict our attention to  clusters associated to reduced expressions of $(w_0,w_0)\in W\times W$ where $w_0$ is the longest word of the Weyl group $W$ of $SL_n$ (in the sense of \cite{BFZ}), then we are able to construct a natural map from the  Poisson spectrum $P.spec(\CC[G])$ to the spectrum $spec(\CC_q[SL_n])$. If we restrict the map to Poisson primitive, resp.~primitive ideals, then it is the inverse of  Yakimov's Dixmier map (see \cite[Sect 4]{Yak-spec} and Section \ref{se:Dixmier map}).

 The paper itself is organized as follows. Section \ref{se:Cluster Algebras} recalls some basic facts about cluster algebras, quantum cluster algebras and compatible Poisson structures. Section \ref{se:COS} provides our main results on the TIPs of quantum cluster algebras, while we discuss in Section \ref{se:spectra} the spectra of quantum cluster algebras and the Poisson spectra of cluster algebras.   Section \ref{se:dbc as P strata}   and Section \ref{se:q-groups} introduce double Bruhat cells and quantum double Bruhat cells, and we show how to construct the Dixmier map from the cluster algebra data. Finally, we added two appendices in which we discuss torus invariant prime ideals in Noetherian rings and properties of quantum and Poisson planes and tori.
 
 \begin{acknowledgements}
  The author would like to thank  Milen Yakimov, Jeanne Scott and Arkady Berenstein for useful discussions about ring theory and cluster algebras on Grassmannians.
 \end{acknowledgements}

\section{Cluster Algebras}
\label{se:Cluster Algebras}
\subsection{Cluster algebras}
 \label {se:cluster algebras def}.
 In this section, we will review the definitions and some basic results on
cluster algebras, or more precisely, on cluster algebras of geometric type over the field of complex numbers $\CC$. Denote by $\mathfrak{F}=\CC(x_1,\ldots,
x_n)$ the field of fractions  in $n$ indeterminates.  Let $B$ be a $m\times n$-integer matrix such that its principal $m\times m$-submatrix is skew-symmetrizable. Recall that a $m
\times
m$-integer matrix $B'$ is called skew-symmetrizable if
there exists a $m \times m$-diagonal matrix  $D$ with positive integer entries  
such that  $B' \cdot D$ is skew-symmetric.
We call the tuple $(x_1,\ldots,x_n, B)$ the {\it initial seed} of the cluster
algebra and   $ (x_1,\ldots x_m)$ a cluster, while ${\bf x}=(x_1,\ldots x_n)$ is
called an extended cluster. The cluster variables $x_{m+1},\ldots,x_n$ are called {\it coefficients}. We will now construct more  clusters, $(y_1,\ldots,
y_m)$ and extended clusters ${\bf
y}=(y_1,\ldots, y_n)$, which are transcendence bases of $\mathfrak{F}$, and the
corresponding 
seeds $({\bf y}, \tilde B)$ in the following way.

For each real number $r$,  denote by $r^+$ and $r^-$ the numbers $r^+={\rm max}(r,0)$ and  $r^-={\rm
min}(r,0)$.
Given a skew-symmetrizable integer $m \times n$-matrix $B$, we define for each
$1\le i\le m$
the {\it exchange polynomial}
\begin{equation}
 P_i = \prod_{k=1}^n x_k^{b_{ik}^+}+ \prod_{k=1}^n  x_k^{-b_{ik}^-}\ .
\end{equation}

We can now define the new cluster variable $x_i'\in\mathfrak{F}$ via the equation
\begin{equation}
\label{eq:exchange}
 x_ix_i'=P_i\ . 
\end{equation}

This allows us to refer to the matrix $B$ as the {\it exchange matrix} of the
cluster $(x_1,\ldots,x_n)$, and to the relations  defined by Equation \ref{eq:exchange} for $i=1,\ldots,m$ as {\it exchange relations}. 

We obtain that $(x_1,x_2,\ldots, \hat x_i,x_i',x_{i+1},\ldots, x_n)$ is a
transcendence basis of $\mathfrak{F}$. We now define the new exchange matrix
$B_i=B'=(b_{ij}')$, associated to the new (extended) cluster $${\bf x}_i=(x_1,x_2,\ldots, \hat
x_i,x_i',x_{i+1},\ldots, x_n)$$
 by defining the coefficients $b_{ij}'$ as follows: 

$\bullet$ $b_{ij}' = -b_{ij}$ if $j \le n$ and $i = k$ or $j = k$,

$\bullet$ $b_{ij}' =  b_{ij} + \frac{|b_{ik} |b_{kj} + b_{ik} |b_{kj} |}{2}$ if
$j \le n$ and $i \ne k$ and $j \ne k$,

$\bullet$ $b_{ij}'=b_{ij}$ otherwise.

This algorithm is called  {\it matrix mutation}. Note that $B_i$ is again
skew-symmetrizable (see e.g.~\cite{FZI}). The process of obtaining a new seed is
called {\it cluster mutation}. The set of seeds obtained from a given seed $({\bf x},B)$ is  called the mutation equivalence class of  $({\bf x},B)$.

\begin{definition}
The cluster algebra $\mathfrak{A}\subset \mathfrak{F}$ corresponding to an
initial seed $(x_1,\ldots, x_n,B)$ is the subalgebra of $\mathfrak{F}$,
generated by the elements of all the clusters in the mutation equivalence class of $({\bf x},B)$ . We refer to the elements of the clusters as the {\it cluster variables}.
\end{definition}

\begin{remark}
Notice that  the coefficients, resp.~frozen variables $x_{m+1},\ldots, x_n$  
will never be mutated. Of course, that explains their name.
\end{remark}

\begin{remark}
	In many cases it is convenient and necessary to assume that the coefficients are invertible elements such as the cluster algebra structure on double Bruhat cells (Section \ref{se:dbc as P strata}). However, this is not so useful  for our project, as we are interested in ideals, and their closure relations, and unnecessary, because $\CC[G]$ admits a cluster algebra structure in our sense.
\end{remark}

We have the following fact, motivating the definition of cluster algebras in the
study of total positivity phenomena and canonical bases.

\begin{proposition} \cite[Section 3]{FZI}(Laurent phenomenon) Let $\mathfrak{A}$
be a cluster algebra with
initial extended cluster $(x_1,\ldots, x_n)$. Any cluster variable $x$ can be
expressed uniquely as a Laurent polynomial in the variables
$x_1,\ldots, x_n$ with integer coefficients.  
\end{proposition}

Moreover, it has been conjectured for all cluster algebras, and proven in many cases (see
e.g.~\cite{MSW} and \cite{FST},\cite{FT})  that the coefficients of these polynomials are
positive.

Finally, we recall the definition of the lower bound of a cluster algebra $\AA$
corresponding to a seed $({\bf x}, B)$. Denote by $y_i$ for $1\le i\le m$ the
cluster variables obtained from ${\bf x}$ through mutation at $i$; i.e., they
satisfy the relation $x_iy_i=P_i$.
 \begin{definition}\cite[Definition 1.10]{BFZ}
\label{def:lower bounds}
Let $\AA$ be a cluster algebra and let $({\bf x}, B)$  be a seed. 
The lower bound $ \mathfrak{L}_B \subset \AA$ associated with $({\bf x}, B)$ is the algebra
 generated by the set  $\{x_1,\ldots x_n,y_1\ldots, y_m\}$.
\end{definition}

\subsection{Upper cluster algebras}
\label{se:upper cluster algebras}
 Berenstein, Fomin and Zelevinsky
introduced the related concept of upper cluster algebras in \cite{BFZ}.  

\begin{definition}
  Let $\mathfrak{A} \subset \mathfrak{F}$ be a
  cluster algebra with initial cluster $(x_1, \ldots, x_n, B)$ and let, as above, $y_1,
  \ldots, y_m$ be the cluster variables obtained by mutation in the directions
  $1, \ldots, m$, respectively.
  
  \noindent(a) The upper bound $\UU_{{\bf x},B} ( \mathfrak{A})$ is defined as
  \begin{equation} 
\UU_{{\bf x},B} ( \mathfrak{A}) = \bigcap_{j = 1}^m \CC [x_1^{\pm 1}, \ldots
     x_{j - 1}^{\pm 1}, x_j, y_j, x_{j + 1}^{\pm 1}, \ldots, x_m^{\pm 1},
x_{m+1},\ldots,x_n] \ . \end{equation}

\noindent(b) The upper cluster algebra $\UU  ( \mathfrak{A})$ is defined as
$$\UU  ( \mathfrak{A})=\bigcap_{({\bf x'},B')}\UU_{\bf x'} ( \mathfrak{A})\ ,$$
where the intersection is over all seeds $({\bf x}',B')$ in the mutation equivalence class of $({\bf x},B)$.
\end{definition}

Observe that each cluster algebra is contained in its upper cluster algebra (see \cite{BFZ}).

\subsection{Poisson structures}
\label{se:poissonstructure}
Cluster algebras are closely related to Poisson algebras. In this section we recall some of the related notions and results. 
 
\begin{definition}
Let $k$ be a field of charactieristic $0$. A Poisson algebra is a pair
$(A,\{\cdot,\cdot\})$ of a commutative $k$-algebra $A$ and a bilinear map
$\{\cdot,\cdot\}:A\otimes A\to A$,  satisfying for all $a,b,c\in A$:
\begin{enumerate}
\item skew-symmetry: $\{a,b\}=-\{b,a\}$ 
\item Jacobi identity: $\{a,\{b,c\}\}+\{c,\{a,b\}\}+\{b,\{c,a\}\}=0$,
\item Leibniz rule: $a\{b,c\}=\{a,b\}c+b\{a,c\}$.
\end{enumerate}
\end{definition} 
If there is no room for confusion we will refer to a Poisson algebra $(A,\{\cdot,\cdot\})$ simply as  $A$.

 Gekhtman, Shapiro and
Vainshtein showed in \cite{GSV} that one can associate Poisson structures to
cluster algebras in the following way.   Let $\mathfrak{A} \subset \CC[x_1^{\pm 1}, \ldots,
x_n^{\pm 1}] \subset \mathfrak{F}$ be a cluster algebra. A Poisson structure
$\{\cdot, \cdot\}$ on
$\CC [x_1, \ldots, x_n]$ is called log-canonical if   $\{
x_i,x_j\}=\lambda_{ij} x_ix_j$ with $\lambda_{ij}\in \CC$ for all $1\le i,j\le n$.
 
The Poisson structure can be naturally extended to $\mathfrak{F}$ by using the identity $0=\{f\cdot f^{-1},g\}$  for all $f,g\in\CC [x_1, \ldots, x_n]$.  We thus obtain that  $\{f^{-1},g\}=-f^{-2}\{f,g\}$ for all $f,g\in \mathfrak{F}$. 
We call $\Lambda=\left( \lambda_{ij}\right)_{i,j=1}^n$ the {\it coefficient matrix} of
the Poisson structure. We say that a Poisson structure on  $\mathfrak{F}$ is compatible with
$\mathfrak{A}$ if it is log-canonical with respect to each cluster $(y_1,\ldots,
y_n)$; i.e., it is log canonical on $\CC[y_1, \ldots, y_n]$.

\begin{remark}
\label{re:Class Poisson}
 A  classification of Poisson structures compatible with cluster algebras was obtained by Gekhtman, Shapiro and Vainshtein in \cite[Theorem 1.4]{GSV}.  It is easy to see from their description that if $n$ is even, then the cluster algebra has an admissible Poisson structure of maximal rank.
\end{remark}

We will refer to  the cluster algebra $\AA$ defined by the initial seeed
$({\bf x},B)$ together with the compatible Poisson structure  defined by the coefficient 
matrix $\Lambda$ with respect to the cluster ${\bf x}$ as the {\it Poisson cluster algebra}
defined by the {\it Poisson seed}  $({\bf x},B,\Lambda)$.

It is not obvious under which conditions  a Poisson seed $({\bf x},B,\Lambda)$ would yield a Poisson bracket $\{\cdot,\cdot\}_\Lambda$   on   $\mathfrak{F}$ such that $\{\AA,\AA\}_\Lambda\subset \AA$. We have, however, the following fact. 

\begin{proposition} 
Let   $({\bf x},B,\Lambda)$ be a Poisson seed and $\AA$ the corresponding cluster algebra. Then $\Lambda$ defines a Poisson algebra structure on the upper bound $\UU_{{\bf x},B}(\AA)$ and the upper cluster algebra $\UU(\AA)$.
\end{proposition}

\begin{proof}
Denote as above by $\{\cdot,\cdot\}_\Lambda$ the Poisson bracket on $\mathfrak{F}$ by $\Lambda$.
  Observe  that  the algebras  $\CC[x_1^{\pm 1},\ldots x_{i-1}^{\pm 1}, x_i,y_i, x_{i+1}^{\pm 1}, \ldots, x_n^{\pm 1}]$   are Poisson subalgebras of the Poisson algebra $\CC[x_1^{\pm 1},\ldots x_n^{\pm 1}]$  for each $1\le i\le m$, as $\{x_i,y_i\}_\Lambda=\{x_i,x_i^{-1}P_i\}_\Lambda\in \CC[x_1,\ldots, x_n]$. If $A$ is a Poisson algebra and $\{B_i\subset A:i\in I\}$ is a family of Poisson subalgebras, then $\bigcap_{i\in I} B_i$ is a Poisson algebra, as well. The assertion follows.
\end{proof}

 \subsection{Compatible Pairs and Their Mutation}
\label{se:Compatible Pairs and Mut}
Section \ref{se:Compatible Pairs and Mut} is dedicated to compatible pairs and their mutation. As we shall see below, compatible pairs yield important examples of Poisson brackets which are compatible with a given cluster algebra structure. Note that our  
 definition is slightly different from the original on in \cite{bz-qclust}. Let, as above, $m\le n$. 
Consider a pair consisting of a skew-symmetrizable $m\times n$-integer matrix
$B$ with rows labeled by the interval $[1,m]=\{1,\ldots, m\}$ and columns  labeled by   $[1,n]$
 together with a skew-symmetrizable $n\times n$-integer matrix  $\Lambda$  with rows and
columns labeled by $[1,n]$.   

\begin{definition}
\label{def:compa pair}
Let $B$ and $\Lambda$ be as above. We say that the pair $(B,\Lambda)$ is
compatible if the coefficients $d_{ij}$ of the $m\times n$-matrix $D=B\cdot \Lambda$ satisfy
$d_{ij}=d_i\delta_{ij}$
for some positive integers $d_i$ ($i\in [1,m]$).  
\end{definition}
This means that $D=B\cdot \Lambda$ is a $m\times n$ matrix where the only
non-zero entries are positive integers on the diagonal of the principal $m\times
m$-submatrix.

The following fact is obvious.

\begin{lemma}
\label{le:full rank}
Let $(B,\Lambda)$ be a compatible pair. Then $B\cdot \Lambda$ has full rank.
\end{lemma}

  Let  $(B,\Lambda)$  be a compatible pair and let $k\in [1,m]$. We define for
$\varepsilon\in \{+1,-1\}$ a $n\times n$ matrix $E_{k,\varepsilon}$ via 
 
 \begin{itemize}
  \item $(E_{k,\varepsilon})_{ij}=\delta_{ij}$ if  $j\ne k$,
  
\item   $(E_{k,\varepsilon})_{ij}= -1$ if   $i=j= k$,
 
\item  $(E_{k,\varepsilon})_{ij}= max(0,-\varepsilon b_{ki})$ if  $i\ne j= k$.

\end{itemize}

 Similarly, we define a $m\times m$ matrix $F_{k,\varepsilon}$ via 
 
  \begin{itemize}
  \item $(F_{k,\varepsilon})_{ij}=\delta_{ij}$   if  $i\ne k$,
  
\item   $(F_{k,\varepsilon})_{ij}= -1$ if   $i=j= k$,
 
\item  $(F_{k,\varepsilon})_{ij}= max(0,\varepsilon b_{jk})$ if  $i= k\ne j$.

\end{itemize}

  We define a new pair $(B_k,\Lambda_k)$ as
  \begin{equation}
  \label{eq:mutation matrix and Poisson}
   B_k=   F_{k,\varepsilon}  B   E_{k,\varepsilon}\ , \quad
\Lambda_k=E_{k,\varepsilon}^T \Lambda E_{k,\varepsilon}\ ,
   \end{equation}
  where $X^T$ denotes the transpose of $X$. We have the following fact.
  
  \begin{proposition}\cite[Prop. 3.4]{bz-qclust}
  \label{pr:comp under mutation}
  The pair  $(B_k,\Lambda_k)$ is compatible. Moreover, $\Lambda_k$ is
independent of the choice of the sign $\varepsilon$. 
  \end{proposition}
 
 The following fact is clear.

\begin{corollary}
Let $\AA$ be a cluster algebra given by an initial seed $({\bf x}, B)$ where $B$ is a $m\times n$-matrix. If $(B,\Lambda)$ is a compatible pair, then $\Lambda$ defines a compatible Poisson bracket on $\mathfrak{F}$ and $\UU(\AA)$.
\end{corollary}

\begin{example}
 If $m=n$ (i.e.~there are no coefficients/frozen variables) and $B$ has full rank, then $(B, \mu B^{-1})$ is a compatible pair for all $\mu\in \ZZ_{> 0}$ such that $\mu B^{-1}$ is an integer matrix.
\end{example}
\begin{remark}
Another important example is the following.  Recall that double Bruhat cells in complex semisimple connected and simply connected algebraic groups have a natural structure of an upper cluster algebra (see \cite{BFZ}). Berenstein and Zelevinsky showed that the standard Poisson structure is given by compatible pairs relative to this upper cluster algebra structure (see \cite[Section 8]{bz-qclust}), see also Sections \ref{se:dbc as P strata} and  \ref{se:q-groups}.  
\end{remark}

    \subsection{Quantum Cluster Algebras}
    \label{se: q-cluster algebra spec}
In this section we recall the definition of a quantum  algebra, introduced by
Berenstein and Zelevinsky in \cite{bz-qclust}. 
\subsection{Quantum Cluster Algebras}
  Recall the definition of a compatible pair from Section \ref{se:Compatible
Pairs and Mut}. 

We define, for each skew-symmetric $n\times n$-matrix $\Lambda$,
the skew-polynomial ring $\CC_\Lambda^t[x_1,\ldots, x_n]$ to be the $\CC[t^{\pm
1}]$-algebra generated by $x_1,\ldots, x_n$ subject to the relations
$$x_ix_j=t^{\lambda_{ij}} x_jx_i\ .$$

Analogously, the quantum torus $H_\Lambda^t$ is defined as the localization of
$\CC_\Lambda^t[x_1,\ldots, x_n]$ at the monoid generated by ${x_1,\ldots, x_n}$,
which is an Ore set. The quantum torus is clearly contained in the skew-field of
fractions $\mathcal{F}_\Lambda$ of $\CC_\Lambda^t[x_1,\ldots, x_n]$, and the
Laurent monomials define a lattice $L\subset  H_\Lambda^t\subset
\mathcal{F}_\Lambda$   isomorphic to $\ZZ^n$. Denote by $x^{1e_1,\ldots,
e_n}$ the monomial $x_1^{e_1}\ldots x_n^{e_n}$.

We are now ready to define the notion of a toric frame in order to define the
quantum cluster algebra.
 
 \begin{definition} 
 A toric frame in $ \mathfrak F$ is a mapping $M:\ZZ^m\to  \mathfrak F-\{0\}$ of
the form
 $$M(c)=\phi(X^{\eta(c)})\ ,$$
 where $\phi$ is a  $\QQ(\frac{1}{2})$-algebra automorphism of $\mathfrak F$ and
$\eta: \ZZ^m\to L$ an isomorphism of lattices.
 \end{definition}
  
  Since a toric frame $M$ is determined uniquely by the images of the standard
basis vectors $\phi(X^{\eta(e_1)})$,\ldots, $\phi(X^{\eta(e_n)})$ of $\ZZ^m$, we
can associate to each toric frame a skew commutative $m\times m$-integer matrix
$\Lambda_M$. 
  We can now  define the quantized version of a seed.
  
  \begin{definition}\cite[Definition 4.5]{bz-qclust}
  A quantum seed is a pair $(M,B)$ where
  \begin{itemize}
  \item $M$ is a toric frame in $\mathfrak F$.
  \item $B$ is a   $m\times n$-integer matrix  with rows labeled by $[1,m]$ and
columns  labeled by an $n$-element subset  $ {\bf ex}\subset [1,m]$.  
  \item The pair $(B,\Lambda_M)$ is compatible.
  \end{itemize}
  \end{definition}
  Now we define the seed mutation in direction of an exchangeable index $k\in 
{\bf ex}$. For each $\varepsilon\in \{ 1,-1\}$ we define a mapping $M_k:
\ZZ^m\to  \mathfrak F$ via
  
  $$M_k(c)=\sum_{p=0}^{c_k} \binom{c_k}{p}_{q^{d_k}{2}} M(E_\varepsilon c
+\varepsilon p b^k)\ , \quad M_k(-c)=M_k(c)^{-1}\ ,$$
  where we use the well-known $q$-binomial coefficients(see e.g.~\cite[Equation
4.11]{bz-qclust}), and the matrix $E_{k,\varepsilon}$ defined in  Section \ref{se:Compatible Pairs and Mut}.
 Define $B_k$ to be obtained from $B$ by the standard matrix mutation in
direction $k$, as in Section \ref{se:cluster algebras def}. One obtains the
following fact.
  
  \begin{proposition}
\cite[Prop. 4.7]{bz-qclust}
(a) The map $M_k$ is a toric frame, independent of the choice of sign
$\varepsilon$. 

\noindent(b) The pair $(B_k,\Lambda_{M_k})$ is a quantum seed. 
  \end{proposition}

Now, given an {\it initial quantum seed} $(B, \Lambda_M) $ denote, in a slight
abuse of notation, by $X_1=M(e_1),\ldots, X_r=M(e_r)$, which we refer to as the
{\it cluster variables} associated to the quantum seed $(M,B)$. Here our
nomenclature differs slightly from \cite{bz-qclust}, since there one considers the
coefficients not to be cluster variables. This is, however, not useful for our
purposes. We now define the seed mutation
$$X_k'=M(-e_k+\sum_{b_{ik}>0} b_{ik}e_i)+ M(-e_k-\sum_{b_{ik}<0} b_{ik}e_i)\ .$$

We obtain that $X_k'=M_k(e_k)$ (see \cite[Prop. 4.9]{bz-qclust}).  
We say that two quantum seeds $(M,B)$ and $(M',B')$ are mutation-equivalent if
they can be obtained from one another by a sequence of mutations. Since
mutations are involutive (see \cite[Prop 4.10]{bz-qclust}), the quantum seeds in
$\mathfrak{F}$ can be grouped in equivalence classes, defined by the relation of
mutation equivalence. The quantum cluster algebra generated by a seed 
$(M,B)\subset \mathfrak F$ is the $\CC[t^{\pm 1}]$-subalgebra generated by
the  cluster variables associated to the seeds in an equivalence class.

\subsection{Toric Actions}
\label{se:toric structure}
We recall the definitions and properties of local and global toric actions
from Gekhtman, Shapiro and Vainsthein \cite{GSV} (see also \cite{GSV B}) where they are introduced in the
context of cluster manifolds. As  discussed in \cite{GSV}, the cluster manifold  associated to a cluster algebra $\AA$ is not necessarily equal to the spectrum of maximal ideals of $\AA$, even when $\AA$ is Noetherian. For example the corresponding variety may have singularities, and hence may not admit a manifold structure.
The main notions, however, carry over into our context.
 
 Let $X$ be an affine  variety such that $\mathfrak{A}=\CC[X]$  is a cluster
algebra or upper cluster algebra. Let ${\bf x}=(x_1,\ldots,x_n)$ be a cluster.  Following \cite[Section
2.3]{GSV} we define for each element
${\bf w}=(w_1,w_2,\ldots, w_n)\in \ZZ^n$ a {\it local toric action} of $\CC^{\ast}$ on
$\CC[x_1,\ldots, x_n]$
via maps $ \psi_{{\bf x},\alpha} : (x_1, \ldots, x_n) \mapsto (\alpha^{w_{ 1}}
x_1, \ldots, \alpha^{w_{n}} x_n)$ for all $\alpha \in \CC^{\ast}$. 
Assume now that we have chosen integer weights ${\bf w}_{\bf x}=(w_1,w_2,\ldots, w_n)$
for each cluster ${\bf x}$.  
The local
toric actions for two clusters are compatible if   the following diagram
commutes for any two clusters ${\bf x} = (x_1,  
  \ldots, x_n)$
and ${\bf y} = (y_1,\ldots, y_n)$, connected by a sequence of
mutations $T$:

$$\begin{xymatrix}{
\CC[{\bf x}]\ar[d]^{\psi_{{\bf x},\alpha}}\ar[r]^T&\CC[{\bf y}]\ar[d]^{\psi_{{\bf
y},\alpha}}\\
\CC[{\bf x}]\ar[r]^T& \CC[{\bf y}]
  }\end{xymatrix}\ .
   $$
 
Compatible local toric actions define a {\it global toric action} on the cluster
algebra and a {\it toric flow} on $X$. We have the following fact.

\begin{lemma}
  \cite[Lemma 2.3]{GSV} 
\label{le:compatible toric} Let $B$ denote the exchange matrix of the
  cluster algebra at the cluster ${\bf x}$. The local toric action  at ${\bf x}$
  defined by ${\bf w}\in \ZZ^n$  can be extended to a global toric action if and only
if $B
  \cdot { {\bf w}} = 0$. Moreover, if such an extension exists, it is unique.
\end{lemma}

We shall now discuss how to obtain all Poisson structures compatible with a
cluster algebra $\mathfrak{A}$, given a   Poisson  seed $({\bf x},B,\Lambda)$ where  $B$ is an  $m\times
 n$-matrix. Denote  $k=n-m$. Let $C$ be an integer $n
\times k$ matrix. We
define an action of the torus $(\CC^\ast)^k$ on $\CC[x_1,\ldots, x_n]$ where ${\bf d}=
(d_1 \ldots, d_k) \in (\CC^{\ast})^k$ acts on $x_i$, $1\le i\le n$, as
\begin{equation}
\label{eq:toric action in class}{\bf d}\cdot_C x_i = x_i \prod_{j = 1}^m d_j^{c_{ij}} . \end{equation}
The local toric action extends to a global toric action of $(\CC^{\ast})^k$ on
${\bf x}$ if and only if $B \cdot C = 0$ by Lemma \ref{le:compatible toric}.
Notice that every
skew-symmetric $k \times k$-matrix $V$ defines a Poisson bracket on.
$(\CC^{\ast})^k$ with $\{x_i, x_j \}_V = v_{ij} x_i x_j$. One
obtains the following result.

\begin{proposition}
  \label{prop:class POisson} \cite[Proposition 2.2]{GSSV}  Let $\UU(\AA)$ be the Poisson upper cluster algebra defined by $({\bf x},B,\Lambda)$, and denote by $\{\cdot,\cdot\}_\Lambda$ the Poisson bracket.
  Let $\{\cdot,\cdot\}'$ be another compatible Poisson structure and let $\{\cdot,\cdot\}'_\lambda$ be the bracket defined by $\{a,b\}'_\lambda=\lambda \{a,b\}'$. Then there exists  a $n \times k$-integer matrix $C$
defining a
  global toric action,  a   skew-symmetric $k \times k $ matrix $V$ and $\lambda\in \CC$ such that
  the  action of Equation \ref{eq:toric action in class} extends to a homomorphism of Poisson algebras
 $$((\CC^{\ast})^m, \{\cdot, \cdot\}_V) \times (\UU(\AA),
     \{\cdot, \cdot\}_\Lambda) \longrightarrow (\UU(\AA),\{\cdot, \cdot\}_\lambda' )\ .$$   
\end{proposition}

\subsection{Toric Actions on Subalgebras}
\label{se:torus on subs}
Let $B$ be an exchange matrix as above and let $T=ker(B)$. Let ${\bf i}=\{x_{i_1},\ldots, x_{i_k}\}$ be a $k$-element subset of ${\bf x}$ and let for $\ell=n-k$ be $\{x_{j_1},\ldots, x_{j_\ell}\}={\bf x}-{\bf i}$. Denote by $\ZZ^{\bf i}$ the sublattice of $\ZZ^n$ spanned by $e_{i_1},\ldots, e_{i_k}$ and by $T_{\bf i}$ the projection  of $T$ on the quotient $\ZZ^n/\ZZ^{\bf i}$. The global toric actions act on $\CC[x_{j_1},\ldots, x_{j_\ell}]$ as follows: Let $t\in T$ and $\alpha \in \CC^\ast$ then
$$t(\alpha)x_{j_h}=t_{\bf i}(\alpha)  x_{j_h}\ , $$
where  $t_{\bf i}$ denotes the image of $t$ under the natural projection of $T$ onto $T_{\bf i}$. Notice that if $B$ is generic, then $rank(T_{\bf i})=max(rank(T), n-|{\bf i}|)$.

 \section{Torus Invariant  Prime Ideals in Quantum Cluster Algebras}
  \label{se:COS}
  In  this section we discuss properties of torus invariant prime ideals (TIPs) in quantum cluster algebras.
  \subsection{Supertoric Clusters}
  Assume that $\AA_q$ is a quantum cluster algebra and that all prime ideals in $\AA$ are completely prime; i.e., if $\II$ is a prime ideal in $\AA_q$, then   $a\cdot b\in \II$ implies $a\in\II$ or $b\in \II$. 
\begin{proposition} 
\label{pr:TPP super toric q}
Let $\AA_q$ be a quantum cluster algebra $({\bf x},B,\Lambda)$ a quantum seed,  and $\II$ a  non-zero two-sided torus invariant prime  ideal. Then, $\II$ contains a   cluster variable $x_i\in {\bf x} $. 
\end{proposition}
 \begin{proof}
 Denote by $\II_{\bf x}$ the intersection of $\II$ with $\CC_\Lambda[x_1,\ldots, x_n]$. Notice first that $\II_{\bf x}\ne 0$. Indeed, let $0\ne f\in \II$. We can express $f$ as a Laurent polynomial in the variables $x_1,\ldots, x_n$; i.e., $f=x_{1}^{-c_1}\ldots x_n^{-c_n} g$ where $c_1,\ldots, c_n\in \ZZ_{\ge 0}$ and $0\ne g\in\CC[x_1,\ldots, x_n]$. Clearly, $g= x_{1}^{c_1}\ldots x_n^{c_n}f \in \II_{\bf x}$. Observe, additionally, that $\II_{\bf x}$ is prime and torus invariant.

We complete the proof  by contradiction. Let $f=\sum_{{\bf w}\in \ZZ^n} c_{\bf w} x^{\bf w}\in\II_{\bf x}$ and suppose that $f$ cannot be factored into $f=gh$ with $g\in\II_{\bf x}$ or $h\in\II_{\bf x}$. We have to show that $f=x_i$ for some $i$. Since the ideal is completely prime, it suffices to show that $f$ is a monomial. We assume that $f$ has the smallest number of nonzero summands such that no monomial term $c_{\bf w} x^{\bf w}$ with $c_{\bf w}\neq 0$ is contained in $\II$. It must therefore have at least two monomial terms.

We need the following fact.

\begin{lemma}
Using the notation introduced above,   a monomial $x^{\bf w}$ with ${\bf w}\in \ZZ^n$ is torus invariant if and only if ${\bf w}\in T^\top$, where $^\top$ denotes the orthogonal complement with respect to the standard bilinear form on $\ZZ^n$.
\end{lemma}
\begin{proof}
Recall from Section \ref{se:toric structure} that if ${\bf b}\in T$ defines a global toric action $ \psi_{{\bf x},\alpha}$, then $x^{\bf w}$ is  invariant under $ \psi_{{\bf x},\alpha}$ if and only if $\sum_{i=1}^n w_ib_i=0$. The assertion follows.
\end{proof}
  
 The function $f$, considered above, must be torus invariant, 
 hence for each pair ${\bf w},{\bf w}'\in \ZZ^n$ with $c_{\bf w},c_{{\bf w}'}\ne 0$ we obtain that ${\bf w}-{\bf w}'={\bf v}\in T^\top$.  Denote by $rad(\Lambda)$ the radical of the skew-symmetric bilinear form; i.e., the set of ${\bf u}\in \ZZ^n$ such that ${\bf u}^T\cdot \Lambda\cdot{\bf u}'=0$ for all ${\bf u}'\in \ZZ^n$. We have the following fact.

\begin{lemma}
The intersection $rad(\Lambda)\cap T^\top=\{0\}$. 
\end{lemma}
\begin{proof}
  
We know that $Im(\Lambda)\otimes\QQ+ker(B)\otimes \QQ= \QQ^n$. Hence $Im(\Lambda)^\top\otimes\QQ\cap T\otimes \QQ^\top=\{0\}$. Notice, that by definition $Im(\Lambda)^\top\otimes\QQ=rad(\Lambda)\otimes \QQ$. The assertion is proved.
\end{proof}

   Assume as above that $c_{\bf w},c_{{\bf w}'}\ne 0$ and ${\bf w}-{\bf w}'={\bf v}\in T^\top$.  The previous lemma yields that ${\bf v}\notin rad(\Lambda)$. This implies that there exists $i\in [1,n]$ such that $\{x_i,x^{\bf v}\}\ne 0$. Therefore,  $\{x_i, x^{\bf w}\}=cx_ix^{\bf w}\ne dx_ix^{{\bf w}'}\{x_i, x^{{\bf w}'}\}$ for some $c,d\in \CC$. 

Clearly, $cx_if-\{x_i,f\}\in \II$ and 
$$ cx_if-\{x_i,f\}=(c-c)c_{\bf w} x^{\bf w}+(c-d)c_{{\bf w}'} x^{{\bf w}'}+\ldots\ . $$

Hence, $cx_if-\{x_i,f\}\ne 0$ and it has fewer monomial summands than $f$ which contradicts our assumption. Therefore,   $\II$ contains a  monomial, and because it is prime it must contain some $x_i\in {\bf x}$. 
  The proposition is proved.
\end{proof}
 
In the following we use notation which is largely self-explanatory, but formally introduced in Appendix \ref{se:Poisson and quantum tori}.  Denote by $T_{[1,k]}$ the projection of the torus $T$ onto the subspace of $\ZZ^n$ defined by $e_{k+1},\ldots, e_n=0$ where $e_i$ denotes the $i$-th standard basis vector. We say that the cluster ${\bf x}$ is super-toric if $rk(T_{[1,k]}+Im(\Lambda_{[1,k]})=k$ for all $k=1,\ldots, n$. We have the following fact.

\begin{proposition}
\label{pr:TPP super toric q+}
 Let $\AA_q$ be a quantum cluster algebra, 
let ${\bf x}$ be a super-toric cluster, and $\II$ a torus invariant non-zero proper prime ideal such that $\II\cap{\bf x}=\{x_{k+1},\ldots, x_n\}$. Then $\II\cap \CC_\Lambda[x_1,\ldots, x_n]$ is generated by  $\{x_{k},\ldots, x_n\}$. 
\end{proposition} 
 
  \begin{proof}
    Suppose that $f\in \CC_{\Lambda_{[1,k-1]}}[x_{1},\ldots, x_{k-1}]\in \II$. The cluster, however,  is super-toric, hence we can adapt the argument from the proof of Proposition \ref{pr:TPP super toric q} to $\CC_{\Lambda_{[1,k-1]}}[x_{1},\ldots, x_{k-1}]$, $T_{[1,k-1]}$ and $\Lambda_{[1,k-1]}$ and show that there exists some $j<k$ such that $x_j \in \II$. We obtain the desired contradiction and the proposition is proved. 
  \end{proof}
  \subsection{The Homomorphism Theorem}
  Let $\AA_q$ be a quantum cluster algebra and let $({\bf x}=(x_1,\ldots, x_n),B,\Lambda)$ be a seed. Assume once again, that all prime ideals in $\AA_q$ are completely prime.  We say that the cluster admits a co-dimension one stratification (COS) of depth $r\le n$ if  there exist proper toric prime ideals $\II_r\subset \II_{n-1}\ldots \subset I_1$ such that $\II_k\cap{\bf x}=\{x_k,\ldots, x_n\}$ for all $1\le k\le r$.
   Denote, as in Appendix \ref{se:Poisson and quantum tori} by $\Lambda_{[1,k-1]}$ the $(k-1)$-st principal submatrix.  We have the following result.

 \begin{theorem}
 \label{th:ideals descr}
  Let $\AA_q$ be a quantum cluster algebra and let ${\bf x}$ be a cluster with a COS of depth $r$. Then there exists an injective algebra homomorphism 
 $$ \AA_q/\II_k\hookrightarrow \CC_{\Lambda_{[1,k-1]}}[x_1^{\pm1},\ldots x_{k-1}^{\pm1}]\ .$$
  
 \end{theorem}

 \begin{proof}
We need the following fact. Let $(\cdot,\cdot)$ denote the standard scalar product on $\ZZ^n$ and let $e_1,\ldots, e_n$ be the standard basis of $\ZZ^n$.
\begin{proposition}
Let $a\in \AA_q$ and let $a=\sum_{{\bf m}\in \ZZ^n} c_{\bf m} x^{\bf m}$ be its Laurent polynomial. Then $a\in \II_k$ if and only 
$({\bf m}, e_i)=0$ for all $i\in[k,n]$ implies that  $c_{\bf m}=0$.
\end{proposition}

 \begin{remark}
 Notice that it is not at all clear that this defines an ideal.
 \end{remark}
 \begin{proof}
 
 We prove the assertion by induction on $(n-k)$. It is trivially satisfied for $n-k=-1$. Suppose now that the theorem   holds for $k+1$.  
We make the  following claim.

\begin{claim}There are injective homomorphisms   of algebras
$$\AA/\II_{k+1}\hookrightarrow  \CC_{\Lambda_k}[x_{1}^{\pm 1},\ldots, x_{k-1}^{\pm 1}, x_{k}]\subset \CC[x_{1}^{\pm 1},\ldots, x_{k}^{\pm1}]\ .$$
\end{claim} 
 
 \begin{proof} The second inclusion is trivial. Let us prove the first one.
 Suppose not. Then there exists $z\in \AA_q/\II_{k+1}$ which can be expressed as $z=x_{k}^{-\ell}\sum_{{\bf m}\in \ZZ^{k}} c_{\bf m} x^{\bf m}$ , where $\ell\in \ZZ_{>0}$, $c_{\bf m}\in \CC$ and where $\ell$ is minimal with the property that $c_{\bf m}\ne 0$ implies that ${\bf m}_{k}\ge 0$. If we multiply by the smallest common denominator (a monomial  $x_{k}^\ell x_1^{\beta_1}\ldots x_{k-1}^{\beta_{k-1}}$ with $\beta_1,\ldots, \beta_{k-1}\in \ZZ_{\ge0}$), then we obtain an element $\tilde z\in  \CC_{\Lambda_{[1,k]}}[x_1,\ldots,x_{k}]\subset\AA_q/\II_{k+1}$. Clearly $\tilde z\in \II_k\subset \AA_q/\II_{k+1}$, where we abuse notation and denote by $\II_k$ the image of the  $\II_k\subset \AA_q$ in $\AA_q/\II_{k+1}$. The element $\tilde z$ contains at least one monomial summand $c_{\bf w} x^{\bf w}$, ${\bf w}\in \ZZ_{\ge 0}^n$ where ${\bf w}_k=0$. We obtain that its pre-image $\tilde z$ under the projection map $\AA_q\to \AA_q/\II_{k+1}$ (which is the identity map on $\CC_{\Lambda_{[1,k]}}[x_1,\ldots, x_{k}]$) lies in $\II_k$, and, hence, as ${\bf x}$ is super-toric, there exists, by Proposition \ref{pr:TPP super toric q+}, a cluster variable $x_i$ with $i\le k-1$ such that $x_i\in \II_k$. That however, contradicts our assumption. The claim is proved.
 \end{proof}

Now, consider the ideal generated by $x_{k}$ in $\CC_{\Lambda_{[1,k]}}[x_{1}^{\pm 1},\ldots, x_{k-1}^{\pm 1}, x_{k}]$  . It is easy to see that it is  torus invariant and prime. Consider its intersection $\tilde \II$ with $\AA/\II_{k+1}$. It suffices to show that it is the unique minimal TIP containing the ideal $\hat \II$  generated by $x_{k}$,  but none of the $x_1,\ldots x_\ell$. It  is, clearly, toric  and prime, because it is the intersection of a  toric    prime ideal  with a torus invariant subring (recall that all prime ideals are assumed to be completely prime). Now, let $z\in \tilde\II-\hat \II$. Then, there exists a monomial $x^{\bf v}$  with ${\bf v}\in \ZZ^k_{\ge 0}$ (e.g.~the smallest common denominator) such that $x^{\bf v}\cdot z\in \hat\II$.  Hence any TIP $\JJ$ not containing  any of the $x_1,\ldots, x_{k-1}$ must contain $z$. This implies that $\tilde \II\subset \JJ$, and hence, $\tilde \II$ is the unique minimal TIP with  $\tilde \II\cap{\bf x}=\{x_{k},\ldots, x_n\}$.  The proposition is proved.
\end{proof}

Theorem \ref{th:ideals descr} now follows from the fact that if $R$ is a  ring, $S\subset R$  a  subring and $I\subset R$ an ideal, then $I\cap S$ is an ideal in $S$, and the canonical inclusion $S/(I\cap S) \hookrightarrow R/I$ is a homomorphism of rings. The theorem is proved.
 \end{proof}
 
 
 \subsection{Semiclassical Versions}
 In \cite{ZW tpc}, we proved versions of the results in the above section regarding torus invariant Poisson prime ideals in cluster algebras, in fact the proofs are almost literally the same. We collect their statements here, for the reader's convenience. In the following, let $\AA$ be a cluster algebra, $({\bf x},B)$   seed and consider a Poisson bracket on $A$ that is obtained from a compatible pair $(B,\Lambda)$. 
 
\begin{proposition} 
\label{pr:TPP super toric cl} Let $\II\subset A$ be  a  non-zero   torus invariant Poisson prime ideal. Then, $\II$ contains a   a cluster variable $x_i\in {\bf x}$.
\end{proposition} 
 Assume from now on that the seed is super-toric.
 
\begin{proposition}  Let   $\II$  be a proper non-zero torus invariant Poisson prime ideal such that $\II\cap{\bf x}=\{x_{k+1},\ldots, x_n\}$. Then $\II\cap \CC_\Lambda[x_1,\ldots, x_n]$ is generated by  $\{x_{k},\ldots, x_n\}$. 
\end{proposition}

  \begin{theorem}
 \label{th:ideals descr cl}
 Under the same assumptions as above, there exists an injective homomorphism of Poisson algebras
 $$ \AA_/\II_k\hookrightarrow (\CC[x_1^{\pm1},\ldots x_{k-1}^{\pm1}], \Lambda_{[1,k-1]})\ .$$  
 \end{theorem} 
 
\section{On the Poisson and Quantum Spectra of Cluster Algebras}
\label{se:spectra}
Recall (e.g.~from Appendix \ref{se: goodearl letzter strat}) that the spectrum of a Noetherian quantum cluster algebra, resp.~ a Noetherian Poisson cluster algebra is stratified by the torus invariant  prime ideals (TIPs), resp.~torus invariant Poisson prime ideals (TIPPs). The following theorem describes these strata. Denote as in Appendix \ref{se: goodearl letzter strat} by $P.spec_\JJ (\AA)$ (resp.~$spec_\JJ (\AA_q)$) the set of  Poisson prime ideals $I$ (resp.~prime ideals) such that $\JJ$ is the maximal TIP $\JJ\subset I$ (resp.~TIPP $\JJ\subset I$).  

\begin{theorem}
Let $(B,\Lambda)$ be a compatible pair and $({\bf x}, B)$ (resp.~($({\bf x}, B,\Lambda)$) a super-toric seed  with COS   of depth $r$ in the corresponding cluster algebra $\AA$ (resp.~quantum cluster algebra $\AA_q$). Then,

\noindent(a) $P.spec_{\II_k}(\AA)$ is homeomorphic to $P.spec(\CC[x_1^{\pm 1},\ldots, x_{k-1}^{\pm 1}],\Lambda)$ and

\noindent(b) $spec_{\II_k} (\AA_q)$ is homeomorphic to $spec(\CC_\Lambda[x_1^{\pm 1},\ldots,x_n^{\pm 1}])$.
  
\end{theorem}

\begin{proof}
Note that the monomials $x^{\bf m}$ with ${\bf m}\in \ZZ_{\ge 0}^{k-1}-\{0\}$ generate   multiplicative sets   in $\AA /\II_k$, resp. $\AA_q\II_k$. Moreover, if $\JJ\supset \II_k$ is a TIPP, resp.~TIP, then we can use the property that the cluster is super-toric to argue, as in the proof of Proposition \ref{pr:TPP super toric q+}, that $\JJ\cap\{x_1,\ldots,x_{k-1}\}\neq \emptyset$. We obtain that  $P.spec_{\II_k}(\AA)$ is homeomorphic to  $P.spec(\AA/\II_k)[x_1^{-1},\ldots,x_{k-1}^{-1}]$, resp.~$spec_{\II_k} (\AA_q)$  to $spec(\AA_q/\II_k)[x_1^{-1},\ldots,x_{k-1}^{-1}]$. But Theorems \ref{th:ideals descr cl} and  \ref{th:ideals descr} yield that $\AA/\II_k[x_1^{-1},\ldots,x_{k-1}^{-1}]=\CC[x_1^{\pm 1},\ldots,x_{k-1}^{\pm 1}]$, resp.~$\AA_q/\II_k[x_1^{-1},\ldots,x_{k-1}^{-1}]=\CC_{\Lambda_{[1,k-1]}}[x_1^{\pm 1},\ldots,x_{k-1}^{\pm 1}]$. The theorem is proved.
\end{proof}

We immediately obtain the following important corollary.

\begin{corollary}
\label{cor:intersection spectra q}
Let $(B,\Lambda)$ be a compatible pair and $({\bf x},B)$, resp.~$({\bf x},B,\Lambda)$ a super-toric seed with COS of depth $r$ in $\AA$, resp.~$\AA_q$. Then the inclusions $\CC[x_1,\ldots, x_n]\hookrightarrow \AA$ and $\CC_\Lambda[x_1,\ldots, x_n]\hookrightarrow \AA_q$ induce inclusion preserving bijective maps 
$$\phi_{q,{\bf x}}:  \bigsqcup_{k=0}^{n+1} spec_{\II_k}(\AA_q) \to spec(\CC_\Lambda[x-1,\ldots,x_n])$$
$$  \phi_{cl,{\bf x}}:  \bigsqcup_{k=0}^{n+1} P.spec_{\II_k}(\AA) \to P.spec(\CC[x-1,\ldots,x_n],\Lambda)\ ,$$
where $\II_0=\{0\}$ and $\II_{n+1}=\AA$, resp.~$\II_{n+1}=\AA_q$.
\end{corollary}

Together with Lemma \ref{le:space homeo} we define the Dixmier-type map  $DIX_{\bf x}$  associated to the cluster ${\bf x}$ by setting $DIX_{\bf x}: \phi_{q,{\bf x}}^{-1}\circ \eta \circ  \phi_{cl,{\bf x}}$ where $\eta$ is the natural  homeomorphism between $P.spec(\CC[x_1,\ldots x_n],\Lambda)$ and $spec(\CC_\Lambda[x_1,\ldots x_n]$ (see also
Lemma \ref{le:space homeo}).

\subsection{Symplectic Leaves}
\label{se:symplectic leaves}
We now consider the case when $\AA$ is Noetherian and hence the coordinate ring of an affine Poisson variety. Given any ideal $I$, we call the maximal Poisson ideal $\mathfrak{P}(I)$ such
that $\mathfrak{P}(I)\subset I$, the {\it Poisson core} of $I$.  It is the sum
of all Poisson ideals contained in $I$. Note that the Poisson core of a prime
ideal is a Poisson prime ideal. 
 
 The Poisson core of a maximal ideal  is called  a  {\it Poisson primitive
ideal}. The Poisson primitive spectrum $P.prim(A)$ is the set of all Poisson
primitive ideals. It is a subset of $P.spec(A)$ and we endow it with the relative
topology. We obtain a continuous, surjective map
 $$ maxspec(A)\to P.prim(A)\ ,$$
 and its fibres are called the {\it symplectic cores}. We obtain a
stratification 
 $$ maxspec(A)=\bigsqcup_{P\in P.prim(A)}\{\mm\in maxspec(A): \PP(\mm)=P\}\ .$$
Poisson and symplectic cores were originally introduced by Brown and Gordon in
\cite{Bro-Go}, and we refer the reader for a more detailed discussion of their 
properties to \cite{Goo1}. In particular it is important to note that symplectic cores stratify into a disjoint union of symplectic leaves, and that they are connected and  smooth in their closure (see \cite[Section 3.3]{Bro-Go}). If a symplectic leaf is equal to a symplectic core, then we call it {\it algebraic}. We have the following fact.
\begin{theorem}
Let $\AA$ be a Noetherian Poisson cluster algebra . Suppose that for each non-zero proper toric Poisson prime ideal $\II$ there exists a supertoric seed $({\bf x}, B)$ with COS of depth $r$ such that $\II=\II_k$ with respect to that subset for some $k\le r$. Then all symplectic leaves are algebraic. 
\end{theorem} 

\begin{proof}
Let $\II_k$ be as above and let $I$ be a Poisson primitive ideal in $\AA$. Denote by $\mathfrak{S}(I)$ its symplectic core. Then the union $ TS(\II_k)=\bigcup_{I\in P.spec_{\II_k}}(S(I))$ is a torus orbit of symplectic cores. It therefore suffices to prove that one of these cores is an algebraic symplectic leaf. Now consider a generic point in  $ TS(\II_k)$; i.e., a point $p$  where $x_1(p),\ldots, x_{k-1}(p)\ne 0$. The dimension of the symplectic core containing $p$ is clearly, $rank(\Lambda_{[1,k-1]})$ which equals the dimension of its symplectic leaf which is a manifold. The symplectic cores are smooth in their closure and connected, and therefore a symplectic manifold. Therefore,  the leaf equals the core and the theorem is proved.    
\end{proof}
 
\section{Cluster Algebras and Double Bruhat Cells}
\label{se:dbc as P strata}
\subsection{Double Bruhat Cells}
 
Let $G$ be a simply connected, connected, semisimple  complex algebraic group  and consider  its   standard Poisson structure. It is well known that the symplectic leaves are connected components of double Bruhat cells (see e.g.~\cite{KZ} and \cite{Yak-spec} and references therein). It is easy to see e.g.~from \cite[Theorem 2.3]{KZ} that the torus orbits of symplectic leaves are closely related to double Bruhat cells. First let us recall the definition of double Bruhat cells and of the upper cluster algebra structure on them, following \cite{BFZ}.

 Now, assume that $G$ has  rank $r$. Let $B$ and $B_-$ be two opposite Borel subgroups while $H=B \cap B_-$ denotes the corresponding maximal torus. The Weyl group is the normalizer of $H$ in $G$, $W=Norm_G(H)/H$, and we denote by $N$ and $N_-$ the unipotent radicals of $B$, resp.~$B_-$. Denote by $\gg=Lie(G)$ the Lie algebra of $G$ and by $\gg=\nn^+\oplus \hh\oplus \nn^-$ the triangular decomposition of $\gg$ corresponding to $B$. Let $R\subset \hh^*$ be the set of roots, $\check R$ the set of coroots and  $\Pi=\{\alpha_1,\ldots,\alpha_r\}\subset \hh^*$  the simple roots. We denote by $\check\alpha_i\in\hh$ the corresponding simple coroots and obtain the Cartan matrix $A=(a_{ij})_{i,j=1}^r$ where $a_{ij}=\alpha_i(\check\alpha_i)$. Recall that $\nn^+$ (resp.~$\nn^-$) has standard basis $e_\alpha$ (resp.~$f_\alpha$), $\alpha\in R$ and we define the standard antisymmetrized $r$-matrix 
$$r_\gg=\sum_{\alpha\in R}\frac{\alpha_i(\check\alpha_i)}{2} e_\alpha\wedge f_\alpha\in \Lambda^2\gg=\Lambda^2 T_e G\ .$$
The standard Poisson structure on $G$ is defined as $\pi_G=L(r_\gg)-R(r_\gg)$ where $L(r_\gg)$ and $R(r_\gg)$ refer to the left-, resp.~ right-invariant bivector fields on $G$ defined by $r_\gg$.

We will next recall some important properties of the Weyl group $W$. The Weyl group is canonically identified (see e.~g.~\cite[Section 2.1]{BFZ}) with the Coxeter group generated by simple reflections $s_i$, $i=1,\ldots,r$ and subject to the relations $(s_is_j)^{d_{ij}}=0$ where
\begin{itemize}
\item $d_{ij}=2$ if $a_{ij}a_{ji}=0$,
\item $d_{ij}=3$ if $a_{ij}a_{ji}=1$,
\item $d_{ij}=4$ if $a_{ij}a_{ji}=2$,
\item $d_{ij}=6$ if $a_{ij}a_{ji}=3$.
 \end{itemize}

A word ${\bf i}=(i_1,\ldots,i_\ell)$ in the alphabet $1,\ldots, r$ is a {\it reduced word} or expression for $w\in W$ if $w=s_{i_1}\ldots s_{i_\ell}$ and $\ell$ is the smallest length of any such factorization. We denote the unique element in the Weyl group with the longest length by $w_0$. We can define the representatives $\overline w\in Norm_G(H)$  for $w\in W$ unambigiously by requiring that 
$\overline u\overline v=\overline {uv}$, if $\ell(u)+\ell(v)=\ell(u+v)$.

We can now define the objects of interest, {\it double Bruhat cells}.
Recall that each Borel subgroup $B$ defines a Bruhat decomposition $G=\bigsqcup\limits_{w\in W} BwB$. Each set $BwB$ is called a {\it Bruhat cell}, and the two opposite Borel subgroups $B$ and $B_-$ define a decomposition  into {\it double Bruhat cells}
$$ G^{u,v}= \bigsqcup\limits_{u,v\in W} (BuB)\cap (B_-v B_-)\ .$$
 Double Bruhat cells can be described as a vanishing set of certain ideals defined in terms of generalized minors, which we introduce in the following. Consider the weight lattice $\Lambda$ of $G$ and its $\ZZ$-basis given by fundamental weights $\omega_i$ for $i=1,\ldots, r$.  Consider the open subset $G_0=N_-H N_+\subset G$ of elements $x\in G$ which have a Gaussian decomposition, denoted as $x=[x]_- [x]_0 [x]_+$ where $[x]_-\in N_-$, $[x]_0\in H$ and $[x]_+\in N_+$. We define, for $u,v\in W$ and $i\in [1,r]$, the {\it generalized minor} $\Delta_{u\omega_i,v \omega_i}$ to be the regular function on $G$ whose restriction to $G_0$ is given by 

$$ \Delta_{u\omega_i,v\omega_i}(x)=\omega_i([\overline u^{-1} x \overline v]_0)  \ .$$

For fixed $i\in[1,r]$, the set of weights $\{w\omega_i:w\in W\}$ is in canonical bijection with the set of cosets $W/Stab_W(\omega_i)$, where $Stab_W(\omega_i)$ is the stabilizer of $\omega_i$ in $W$. The restriction of the Bruhat order to the set of minimal coset representatives induces a partial order on $\{w\omega_i:w\in W\}$, which is also referred to as the  Bruhat order. Note that the function $ \Delta_{u\omega_i,u\omega_i}$ is dependent only on the weights $u\omega_i$ and $v\omega_i$, not on the particular choices of $u$ and $v$ (see e.~g.~\cite{BFZ}). We have the following description of double Bruhat cells

\begin{proposition}\cite[Proposition 2.8]{BFZ}
\label{pr:DBC-def}
The double Bruhat cell $G^{u,v}\subset G$ is given by the following conditions: For each $i\in[1,r]$

\noindent(a) $\Delta_{u'\omega_i,\omega_i}=0$ if $u'\omega_i\not\le u\omega_i$ 

\noindent(b)$\Delta_{\omega_i,v'\omega_i}=0$ if $v'\omega_i\not\le v^{-1}\omega_i$

\noindent(c) $\Delta_{u\omega_i,\omega_i}\ne 0$ and $\Delta_{\omega_i,v^{-1}\omega_i}\ne 0$.
\end{proposition}

\subsection{(Upper) Cluster Algebra Structure on Double Bruhat cells}
Berenstein, Fomin and Zelevinsky define in \cite{BFZ} an upper cluster algebra structure with invertible coefficients on double Bruhat cells as follows. Consider the Coxeter group $W\times W$. We will label the generators of the first copy by $s_{-1},\ldots s_{-r}$ and the generators of the second copy by $s_{1},\ldots s_{r}$. Reduced words for a pair $(u,v)\in W\times W$ correspond to arbitrary shuffles of reduced expressions for $u$ in the $s_{-1},\ldots s_{-r}$ and for $v$ in the set $s_{1},\ldots s_{r}$. We will construct for each reduced expression ${\bf i}$ of an element $(u,v)$ a cluster consisting of generalized minors and an exchange matrix. 

\subsubsection{The cluster}
For each $k\in [1,\ell(u)+\ell(v)]$ define 
$$
u_{\le k}=u_{\le k}({\bf i})=\prod\limits_{\ell=1,\ldots k;\  \varepsilon(i_\ell)=-1}s_{|i_\ell|}\ ,
$$
$$
v_{>k}=v_{>k}({\bf i})=\prod\limits_{\ell=\ell(u)+\ell(v),\ldots k+1;\  \varepsilon(i_\ell)=1}s_{|i_\ell|}\ ,
$$
with increasing indices in the case of $u_{\le k}$, and decreasing indices in the second case. Employing the convention that $u_{\le k}=e$  and $v_{>k}=v^{-1}$ for $k\le 0$, we set for $k\in [-r,-1]\cup[1, \ell(u)+\ell(v)]$ 
$$\Delta(k,{\bf i})=\Delta_{u_{\le k}\omega_i,v_{>k}\omega_i}\ .$$

We now have to construct the exchange matrix for the cluster variables $x_k= \Delta(k,{\bf i})$. For each index $k$ denote by $k^+$ the minimal number $\ell$ such that $k<\ell$ and $|i_k|=|i_\ell|$. If $|i_k|\neq|i_\ell|$ for all $\ell>k$, then we set $k^+= \ell(u)+\ell(v)+1$. The exchangeable variables will be those $x_k$ for which $k>0$ and $k^+\le \ell(u)+\ell(v)$. This implies that the variables with negative indices are coefficients, as well as those whose index represents the right-most occurrence of $|i_k|$.  Notice that in order to adapt the cluster algebras to the notation developed in Section \ref{se:Cluster Algebras}, one may need to permute the indices of the cluster variables. 
Note, additionally, that the coefficients will appear at different indices if we choose different shuffles or reduced expressions for $(u,v)$.
\begin{remark}
\label{re:coeffs}
The coefficients in $\CC[G^{u,v}]$ are the minors $\Delta_{u\omega_i,\omega_i}$and $\Delta_{\omega_i,v^{-1}\omega_i}$ which appear in the non-vanishing condition \ref{pr:DBC-def}(c).
\end{remark}
Berenstein, Fomin and Zelevinsky give two different ways to determine the exchange matrix in \cite[Sect. 2.2]{BFZ}. For our purposes it will be most convenient to follow the more involved, but more concrete, path by assigning a directed graph $\Gamma({\bf i})$ to the set $ [-r,-1]\cup[1, \ell(u)+\ell(v)]$.  Assume from now on that $k<\ell$, and assume that $k$ or $\ell$ are  exchangeable. The vertices are the set of indices $[-r,-1]\cup[1, \ell(u)+\ell(v)]$. Two vertices $k$ and $\ell$ are connected by a path if $k$ or $\ell$ are {\bf i}-exchangeable and one of the following are satisfied.

\begin{enumerate}
 \item {\it Edge from $k$ to $\ell$} if  one of the following holds:
      \begin{enumerate}
\item $\ell=k^+$ and $\varepsilon(i_\ell)=1$,
\item $\ell<k^+<\ell^+$, $a_{|i_k|,|i_\ell|}\ne 0$ and $-1=\varepsilon(i_\ell)=\varepsilon(i_{k^+})$,
\item $\ell<\ell^+<k^+$,$a_{|i_k|,|i_\ell|}\ne 0$ and $-1=\varepsilon(i_\ell)=-\varepsilon(i_{\ell^+})$.
       \end{enumerate}
\vspace{2mm}

\item{\it Edge from $\ell$ to $k$} if one of the following holds:
 \begin{enumerate}
\item $\ell=k^+$ and $\varepsilon(i_\ell)=-1$,
\item $\ell<k^+<\ell^+$, $a_{|i_k|,|i_\ell|}\ne 0$ and $1=\varepsilon(i_\ell)=\varepsilon(i_{k^+})$,
\item $\ell<\ell^+<k^+$,$a_{|i_k|,|i_\ell|}\ne 0$ and $1=\varepsilon(i_\ell)=-\varepsilon(i_{\ell^+})$.
       \end{enumerate}
\end{enumerate}

We can now define the exchange matrix $B$ as the integer matrix with rows labeled by the indices $ [-r,-1]\cup[1, \ell(u)+\ell(v)]$ and columns labeled by the  exchangeable indices. The coefficients are given by the following rules.
\begin{enumerate}
 \item The coefficient $b_{k\ell}\ne 0$ if and only if there is an edge connecting $k$ and $\ell$ and $b_{k\ell}> 0$ if and only if the edge is directed from $k$ to $\ell$. 

\item If $b_{k\ell}\ne 0$, then

\begin{itemize}
\item $|b_{k\ell}|=  1$ if $|i_k|=|i_\ell|$,
\item  $|b_{k\ell}|=    -a_{|i_k|,|i_\ell|} $ if  $|i_k|=|i_\ell|$.
\end{itemize}
 \end{enumerate}

Denote by $\AA^{u,v}({\bf i})$ the cluster algebra defined by the data above. The cluster algebra structure defined in \cite{BFZ} is the corresponding cluster algebra with invertible coefficients. It is proved in \cite{BFZ} that the associated upper cluster algebra $\AA^{u,v}$ is isomorphic as an algebra to the algebra of functions on $G^{u,v}$.

\begin{example}
Let $G=SL_3(\CC)$. The generalized minors coincide with the regular matrix minors in the case of $GL_n$.  Consider $G^{w_0,w_0}$  and the initial seed corresponding to the shuffle $(w_0,w_0)=s_1 s_2s_1 s_{-1}s_{-2} s_{-1}$. We obtain the following graph associated to the variables of the initial seed:
$$ 
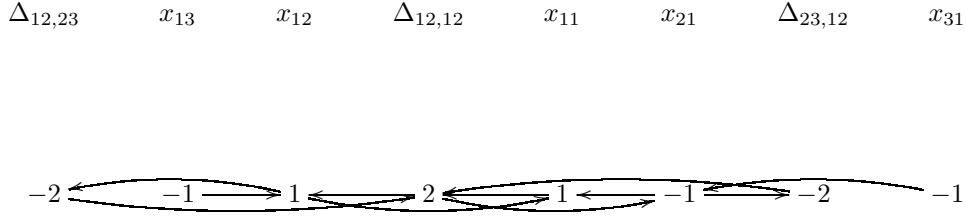
\xymatrix{  \Delta_{12,23} & x_{13} & x_{12} &\Delta_{12,12} &x_{11} &x_{21} &\Delta_{23,12} &x_{31} \\
\\
   -2\ar@/_/[rrr]&-1\ar[r]&1\ar@/_/[ll]\ar@/_/[rr]&2\ar[l] \ar@/_/[rr]&1 \ar[l]&-1 \ar[l]\ar[r]&-2 \ar@/_/[lll]&-1 \ar@/_/[ll]} $$
 
where $\Delta_{ij,k\ell}$ denotes the $2\times 2$-minor $\Delta_{ij,k\ell}=x_{ik}x_{j\ell}-x_{jk} x_{i\ell}$. The coefficients are $\Delta_{12,23}$, $x_{13}$,$\Delta_{23,12}$ and $x_{31}$.  Our example is the running example of \cite{BFZ}. For more details see \cite[Example 2.9]{BFZ}.

\end{example}

\subsubsection{Compatibility with torus actions and Poisson structure}
Consider the double  Bruhat cell $G^{u,v}$ with cluster seed as defined above. It is well known that the exchange polynomials are homogeneous with respect to the weight grading. It is also well known  that the action of the torus $T$ defines global toric actions. Moreover, we recall the definition of the Poisson bracket on generalized minors (see \cite{KZ}) in the standard cluster:
\begin{equation}
\label{eq:Poisson Bruhat}
\{\Delta_{\gamma,\delta},\Delta_{\gamma',\delta'}\}=\left((\gamma,\gamma')-(\delta,\delta')\right)\Delta_{\gamma,\delta},\Delta_{\gamma',\delta'} .\end{equation}
In order to verify that the Poisson structure is indeed compatible it suffices to verify, following \cite[Section 2]{GSV}, that if $x_i$ is an exchangeable variable, then the cluster $(x_1,\ldots, \hat x_i,y_i,x_{i+1},\ldots, x_{r+\ell(u)+\ell(v)})$ is a system of log-canonical coordinates. Here we use the convention that  $y_i=x_i^{-1} P_i$, where $P_i$ is the  exchange polynomial.
In our case we obtain  that the exchange polynomial $P_i$ is homogeneous with respect to the weight grading, and henceforth, $\{P_i,x_j\}= t P_ix_j$ by \ref{eq:Poisson Bruhat} for some $t\in \CC$. This immediately implies that $(x_1,\ldots, \hat x_i,y_i,x_{i+1},\ldots, x_{r+\ell(u)+\ell(v)})$ is a system of log-canonical coordinates. The standard Poisson structure is, therefore, compatible with the cluster algebra structure.

The following, well known,  fact motivates our whole research project.
\begin{proposition}
The upper cluster algebra $\AA^{w_0,w_0}$, with noninvertible coefficients, is isomorphic to the algebra of functions $\CC[G]$.
\end{proposition}

\section{Quantum Groups and Quantum Double Bruhat Cells}
\label{se:q-groups}
In this section we recall some well-known facts about quantum groups, their torus invariant prime ideals and the corresponding quantized coordinate rings. 
\subsection{The Quantized Enveloping Algebra $U_q(\gg)$}
We start with the definition of the quantized enveloping algebra
associated with a complex reductive Lie algebra $\gg$ (our standard
reference here will be \cite{brown-goodearl}). Let $\hh\subset \gg$
be a Cartan subalgebra,  $P(\gg)$ the weight lattice, as introduced above, and let $A=(a_{ij})$ be the Cartan matrix
for $\gg$. Additionally, let $(\cdot,\cdot)$ be the standard non-degenerate symmetric bilinear form on $\hh$.

The {\it quantized enveloping algebra} $U_q(\gg)$ is a $\CC(q)$-algebra generated
by the elements $E_i$ and $F_i$ for $i \in [1,r]$, and $K_\lambda $ for $\lambda \in P(\gg)$,
subject to the following relations:
$K_\lambda  K_\mu = K_{\lambda +\mu}, \,\, K_0 = 1$
for $\lambda , \mu \in P$; $K_\lambda E_i =q^{(\alpha_i\,,\,\lambda)} E_i K_\lambda,
\,\, K_\lambda F_i =q^{-(\alpha_i\,,\,\lambda)} F_i K_\lambda
$
for $i \in [1,r]$ and $\lambda\in P$;
\begin{equation}
\label{eq:upper lower relations}
E_i,F_j-F_jE_i=\delta_{ij}\frac{K_{\alpha_i}- K_{-\alpha_i}}{q^{d_i}-q^{-d_i}}
\end{equation}
for $i,j \in [1,r]$, where  $d_i=\frac{(\alpha_i\, ,\,\alpha_i)}{2}$;
and the {\it quantum Serre relations}
\begin{equation}
\label{eq:quantum Serre relations}
\sum_{p=0}^{1-a_{ij}} (-1)^p 
E_i^{(1-a_{ij}-p)} E_j E_i^{(p)} = 0,~\sum_{p=0}^{1-a_{ij}} (-1)^p 
F_i^{(1-a_{ij}-p)} F_j F_i^{(p)} = 0
\end{equation}
for $i \neq j$, where
the notation $X_i^{(p)}$ stands for the \emph{divided power}
\begin{equation}
\label{eq:divided-power}
X_i^{(p)} = \frac{X^p}{(1)_i \cdots (p)_i}, \quad
(k)_i = \frac{q^{kd_i}-q^{-kd_i}}{q^{d_i}-q^{-d_i}} \ .
\end{equation}

%

The algebra $U_q(\gg)$ is a $q$-deformation of the universal enveloping algebra of
the reductive Lie algebra~$\gg$.
It has a natural structure of a bialgebra with the co-multiplication $\Delta:U_q(\gg)\to U_q(\gg)\otimes U_q(\gg)$
and the co-unit homomorphism  $\varepsilon:U_q(\gg)\to \CC(q)$
given by
\begin{equation}
\label{eq:coproduct}
\Delta(E_i)=E_i\otimes 1+K_{\alpha_i}\otimes E_i, \,
\Delta(F_i)=F_i\otimes K_{-\alpha_i}+ 1\otimes F_i, \, \Delta(K_\lambda)=
K_\lambda\otimes K_\lambda \ ,
\end{equation}
\begin{equation}
\label{eq:counit}
\varepsilon(E_i)=\varepsilon(F_i)=0, \quad \varepsilon(K_\lambda)=1\ .
\end{equation}
In fact, $U$ is a Hopf algebra with the antipode anti-homomorphism $S: U \to U$ given by
\begin{equation}
\label{eq:antipode}
 S(E_i) = -K_{-\alpha_i} E_i, \,\, S(F_i) = -F_i K_{\alpha_i}, \,\,
S(K_\lambda) = K_{-\lambda}\ .
\end{equation}

Let $U_q(\gg)^-$ (resp.~$U_q(\gg)^0$; $U_q(\gg)^+$) be the $\CC(q)$-subalgebra of~$U_q(\gg)$ generated by
$F_1, \dots, F_r$ (resp. by~$K_\lambda \, (\lambda\in P)$; by $E_1, \dots, E_r$).
It is well-known that $U_q(\gg)=U_q(\gg)^-\cdot U_q(\gg)^0\cdot U_q(\gg)^+$ (more precisely,
the multiplication map induces an isomorphism of vector spaces $U_q(\gg)^-\otimes U_q(\gg)^0\otimes U_q(\gg)^+ \to U$). Sometimes we will refer to $U_q(\gg)^{\pm}$ as $U_q(\nn^\pm)$.



In order to study the finite dual of $U_q(\gg)$ in the following section, we need to consider the full sub-category  $\OO_{f}$ of the category
$U_q(\gg)-Mod$. The objects of $\OO_{f}$ are finite-dimensional
$U_{q}(\gg)$-modules $V^q$ having a weight decomposition
$$V^q=\oplus_{\mu\in P} V^q(\mu)\ ,$$
where each $K_\lambda$ acts on each {\it weight space} $V^q(\mu)$ by
the multiplication with $q^{(\lambda\,|\,\mu)}$ commonly referred to as type I modules (see e.g.,
\cite{brown-goodearl}[I.6.12]). The category $\OO_{f}$ is semisimple
and the irreducible objects $V^q_\lambda$ are  generated by highest
weight spaces $V^q_\lambda(\lambda)=\CC(q)\cdot v_\lambda$, where
$\lambda$ is a {\it dominant weight}, i.e, $\lambda$ belongs to
$P^+=\{\lambda\in P:(\lambda\,|\,\alpha_i)\ge 0~ \forall ~i\in
[1,r]\}$, the monoid of dominant weights.

\subsection{The Quantum Weyl Group and the Algebras $U_q(\nn^+_w)$}
 Denote by $W$ the Weyl group of $\gg$ generated by the simple reflections  $s_i$ for $i\in[1,r]$. Corresponding to each $i\in [1,r]$ there exists a Hopf algebra-automorphism $T_{i}:U_q(\gg)\to U_q(\gg)$ defined on the generators of $U_q(\gg)$ in the following way:

\begin{equation}
\label{eq: T_i on Ui}
T_{i}(E_{i})=-F_{i} K_{\alpha_{i}}\ ,T_{i}(F_{i})=-K_{\alpha_{i}}^{-1}E_{i}\ , T_{i}(E_{j})=\sum_{k=0}^{-a_{ij}} (-1)^{k-a_{ij}}q_{i}^{-k}  E_{i}^{-a_{ij}-k}E_{j}E_{i}^{k}\ ,
\end{equation}
$$T_{i}(F_{j})=\sum_{k=0}^{-a_{ij}} (-1)^{k-a_{ij}}q_{i}^{k} F_{i}^{k}F_{j}F_{i}^{a_{ij}-k}\ , T_{i}(K_{\lambda})=K_{\sigma_{i}(\lambda}\ .$$ 
The action of the $T_{i}$ on $U_q(\gg)$ corresponds to the adjoint action of the simple reflection $s_{i}$ on the universal enveloping algebra $U(\gg)$. For every element $w\in W$ with presentation $w=w_{i_{1}}\ldots s_{i_{k}}$ we define $T_{w}$ as $T_{w}=T_{s_{i_{1}}}\cdots T_{s_{i_{k}}}$. 
We will need the following well known fact.
 \begin{lemma}\cite[8.18]{Jan} 
 If $w\in W$, then $T_w$ is independent of the choice of reduced expression; i.e. if $w=s_{i_1}\ldots s_{i_k}$ and  $w=s_{j_1}\ldots s_{j_k}$ are reduced expressions of $w\in W$, then 
 $$T_{s_{i_1}}\ldots T_{s_{i_k}}=T_{s_{j_1}}\ldots T_{s_{j_k}}\ .$$
 \end{lemma}

We  define for each presentation of $w_{0}=w_{i_{1}}\ldots s_{i_{r}}$  a set of  positive roots spanning $U^{+}$    following \cite[ch.8]{Jan}:
$$E_{\alpha_{(1)}}= E_{i_{1}}\ ,~E_{\alpha_{(2)}}=T_{i_{1}}(E_{i_{2}})\ ,~\ldots E_{\alpha_{(k)}}= T_{i_{1}}\cdots T_{i_{r-1}}(E_{i_{r}}) \ .$$
 
We now define for $w\in W$ the algebras $U_q(\nn_w^+)$ as follows. Let $w=w_{i_{1}}\ldots s_{i_{k}}$ define $U_q(\nn_w^+)$ to be the algebra generated by $E_{\alpha_{(1)}}= E_{i_{1}}$, $E_{\alpha_{(2)}}=T_{i_{1}}(E_{i_{2}})$, \ldots $E_{\alpha_{(k)}}= T_{i_{1}}\cdots T_{i_{k-1}}(E_{i_{k}})$.  The following fact is well known.

\begin{lemma}(see e.g. \cite[Section 7.1]{GLSq}
The algebra $U_q(\nn_w^+)$ is independent of the reduced expression for $w\in W$. 
\end{lemma}

\subsection{The Quantum Group $R_q[G]$}
In this section we will discuss the algebra of functions on the quantum group $U_q(\gg)$. Recall the definition of the finite dual  $H^\ast$ of a Hopf algebra $H$ inside the dual of $H$.
\begin{definition}
The finite dual $H^\ast$ of a Hopf algebra $H$ is defined as the set of all $h\in H^*$ for which $h(I)0=$ for some ideal $I\subset H$.
\end{definition} 

We have the following fact.

\begin{lemma} \cite[Chapter 1.4]{Jo}
$H^\ast$ is a Hopf algebra.
\end{lemma}
The Hopf algebra $H^\ast$  has a description in terms of finite-dimensional modules of $H$. 
Let $V$ be a finite-dimensional left $H$-module and $V^*$ its dual. For each $\xi \in V^*$  and $v\in V$  one  defines the coordinate function $c_{\xi,v}^V\in H^*$ via $c_{\xi,v}^V(h)=\xi(h.v)$ for $h\in H$. Denote by $C^V$ the linear span of the functions $c_{\xi,v}^V$ for all $\xi \in V^*$  and $v\in V$.  We have the following fact.

\begin{lemma}\cite[Chapter 1.4]{Jo}
Let $F$ be the subspace of $H^*$ spanned by the $C^V$ where $V$ runs over the finite-dimensional $H$-modules. Then

\noindent (a) $F=H^\ast$.

\noindent (b) $C^V+C^{V'}=C^{V\oplus V'}$ and $C^V \cdot C^{V'}=C^{V\otimes V'}$.
\end{lemma}

For more details regarding the Hopf algebra structure we will refer the reader to Joseph's book  \cite{Jo}. In the case when $H=U_q(\gg)$ we will refer to $H^\ast$ as $H^\ast=R _q(G)$, where $G$ is the simply connected, connected semisimple algebraic group with Lie algebra $\gg$. 

 \subsection{$U_q(\nn^+)^\ast$ and $U_q(\nn_w^+)^\ast$}
  In order to give a representation theoretic construction of $U_q(\nn)^\ast$ we cannot rely on the use of the finite dual, as $U_q(\nn^+)$ is not a Hopf algebra, however, it has a natural algebra structure and is isomorphic as an algebra to $U_q(\nn^+)$ (see e.g. \cite{GLSq} and \cite{Yak-spec}). Analogously, the  algebra $U_q(\nn_w^+)^\ast$ is defined as the image of $U_q(\nn_w^+)$ under the isomorphism mentioned above. The following fact was proved by Gei\ss, Leclerc and Schr\"oer.
  
  \begin{theorem}\cite[Theorem 1.1]{GLSq}
  Let $\gg$ be as above, and assume that the Cartan matrix $A$ is symmetric. Then $U_q(\nn_w^+)^\ast$  admits a quantum cluster algebra structure. for all $w\in W$. 
  \end{theorem}

  \subsection{Generalized Quantum Minors and Quantum Double Bruhat Cells }
 \label{se:Gqm and qdbc}
 In this section we recall the notion of generalized quantum minors which were originally defined by Berenstein and Zelevinsky in \cite[Section 9]{bz-qclust}. For each dominant weight $\lambda\in P^+$,   define $\Delta^\lambda \in R_q[G]$ by
 $$ \Delta^\lambda(FK_\mu E)= \varepsilon(F) q^{(\lambda|\mu)} \varepsilon(F)  \ ,$$
 where  $F\in U^-$, $\mu\in P$ and $E\in U^+$.  Now, let  $(u,v)\in W\times W$ and let  $u=s_{i_1}\ldots s_{i_k}$ and $v=s_{j_1}\ldots s_{j_\ell}$ be reduced expressions for $u$ and $v$. Define  by $\check\eta_{g}=s_{i_1}\ldots s_{i_g}(\check\alpha_{i_g})$ and by $\check\zeta_{h}=s_{j_1}\ldots s_{j_h}(\check\alpha_{j_h})$ for $1\le g\le k$ and $1\le h\le \ell$.  Denote by 
 $$T^+(u,\lambda)=E_{i_1}^{\left((\lambda(\check\eta_1)\right)} \ldots E_{i_k}^{\left((\lambda(\check\eta_k)\right)}\ ,
 T^-(v,\lambda)=F_{j_\ell}^{\left((\lambda(\check\zeta_\ell)\right)} \ldots F_{j_1}^{\left((\lambda(\check\zeta_1)\right)}\ .$$

 Now, we define  the {generalized quantum minor} $\Delta_{u\lambda, v\lambda}$ by its action on $x\in U_q(\gg)$ as 
 $$\Delta_{u\lambda, v\lambda}(x)=\Delta^\lambda\left(T^+(u,\lambda) x T^-(v,\lambda)\right) \ .$$
 Indeed, $\Delta_{u\lambda, v\lambda}$ is well defined since its definition does not depend on the choice of reduced expressions for $u$ and $v$ (see \cite[Section 9.3]{bz-qclust} ). 
 
We can now define the quantum double Bruhat cells. Denote by $U^0$ the subalgebra of $U_q(\gg)$ generated by the $K_\lambda$, $\lambda\in P$. Then, define  $U_i^+$ to be  the subalgebra generated by $U^0$ and $E_i$, wherea  $U_i^-$ is the subalgebra generated by $U^0$ and $F_i$. 
Now let $(u,v)\in W\times W$ and, as above, let $u=s_{i_1}\ldots s_{i_k}$ and $v=s_{j_1}\ldots s_{j_\ell}$ be reduced expressions.  Now we define
$$U_{u,v}=U_{i_1}^+ \ldots U_{i_k}^+U_{j_1}^-\ldots U_{j_\ell}^-\subset U_q(\gg)\ ,$$
and the ideal $J_{(u,v)}\subset R_q[G]$ by $f\in J_{(u,v)}$ if $f(U_{u,v})=0$.  The ideal is indeed two-sided since $\Delta(U_{u,v})\subset U_{u,v}\otimes U_{u,v}$ (see \cite[Section 9.3]{bz-qclust} ) . 
The {\it closed quantum double Bruhat cell} is now defined as the quotient  $R_q[G]^{u,v}=R_q[G]/J_{(u,v)}$. 
The torus orbits of symplectic leaves on the semisimple group $G$ correspond to the open double Bruhat cells, therefore it is useful to recall the definition of  their quantum analogs from \cite[Section 9.3]{bz-qclust} . One introduces a denominator set $D_{u,v}$ as
$$\{\Delta_{u\lambda, \lambda} \Delta_{\mu, v^{-1}\mu}:\ k\in \ZZ;\  \lambda,\mu\in P^+\}$$ 
and defines $\OO_q(G)^{u,v}=R_q[G]/J_{(u,v)}[D_{u,v}^{-1}]$, the {\it open quantum double Bruhat cell}.

Recall that an element $u$ of a ring $R$ is called normal if $uR=Ru$. 
It was proved by Joseph that the torus invariant normal elements in the double Bruhat cells $R_q[G]^{u,v}$ are , up to scalar multiplication,  elements of  $D_{u,v}$ (see \cite[Sec 2]{Yak-spec}). We have the following fact.

\begin{lemma}\cite[Sec 2]{Yak-spec}
\label{le:center is coeffs}
The center of $\OO_q(G)^{u,v}$ is a Laurent polynomial ring $\CC[c_1^{\pm 1},\ldots c_s^{\pm 1}]$ where $c_1,\ldots c_s\in D_{u,v}$.
\end{lemma}
 
 \begin{remark}
Observe that \cite[Section 9.3]{bz-qclust} implies that the Goodearl Letzter Stratification Theory  can be applied to these algebras (see \ref{se: goodearl letzter strat} and \cite[Section II.4.4]{brown-goodearl}). 
\end{remark}  

Berenstein and Zelevinsky conjectured in \cite{bz-qclust} that $\OO_q(G)^{u,v}$ and $R_q[G]$ have a quantum cluster algebra structure with initial seeds given as in the classical case by quantum minors. It follows from the  work of Gei\ss\ Leclerc and Schr\"oer \cite[Section 12]{GLSq} that we do have the cluster algebra structures in the case of $G=SL_n$, as the quantum cluster algebra structure on the quantum $n\times n$-matrices induces a quantum cluster algebra structure on $R_q[SL_n]$.

 
 


\subsection{The Dixmier Map}
\label{se:Dixmier map}
In this section let $G=SL_n$. However, we will refer to it as $G$ to indicate that all results and conjectures also apply in the case when $G$ is an arbitrary complex semisimple algebraic group.
 Let  ${\bf w}_0=s_{i_1}\ldots s_{i_k}$ and ${\bf w}_0'=s_{i_1}\ldots s_{i_k}$  be  reduced expressions for $w_0$, and let  ${\bf x}= {\bf x}_{{\bf w}_0, {\bf w}_0'}$ be the corresponding cluster cluster consisting of generalized (quantum) minors. Consider a  sequence of pairs    $(e,e)=(u_0,v_0),(u_1,v_1),\ldots, (u_2,v_2)\ldots (u_{2k},v_{2k})=(w_0,w_0)$ with $(u_i,v_i)\in W\times W$  and $(u_i,v_i)<(u_{i+1},v_{i+1})$ in the order induced by the Bruhat order on $W\times W$. It is easy to see that $J_0=J_{e,e}\supset J_{(u_1,v_1)}\supset\ldots\supset J_{u_{2k},v_{2k}}\supset R_q[G]$ defines a COS of depth $r$.   
The quotients are by definition the corresponding closed quantum double Bruhat cells. Of course,  we obtain analogous patterns in the classical case.
 
 We are now ready  to construct the Dixmier map.

 \begin{proposition}
 \label{pr:Dixmier type map}
 Let $W_0$ be the set of reduced expressions for $w_0\in W$. Then the set of Dixmier type maps $DIX_{\bf x}$ for the clusters ${\bf x}={\bf x}_{{\bf w}_0, {\bf w}_0'}$ where  $({\bf w}_0, {\bf w}_0')$ runs over the elements of $W_0\times W_0$, defines a bijection between $spec(R_q[G])$ and $P.spec(\CC[G])$. 
 \end{proposition}
 
 \begin{proof}
 It only remains to verify that the map is well defined: i.e. that the image of $spec(R_q[G]^{u,v})$ under a Dixmier type map does not depend on $({\bf w}_0, {\bf w}_0')$. However,   Lemma \ref{le:center is coeffs} and Remark \ref{re:coeffs} imply that the center of $R_q[G]^{u,v}$ is induced by elements of the Laurent polynomial ring in the coefficients in the upper cluster algebra structure, and hence is independent of   the choice of reduced expression.  \end{proof}
 
 \begin{remark}
 By the construction of the double Bruhat cells, Proposition \ref{pr:Dixmier type map} also holds for arbitrary $G$. Indeed, we do not need the cluster algebra structure at all for this result, it follows straight from Yakimov's \cite{Yak-spec}. 
 \end{remark}
 
 \subsubsection{Topological Properties of the Dixmier map}
 Using Section \ref{se:symplectic leaves} we may restrict the map to primitive ideals in the quantum case and Poisson primitive ideals, resp.~symplectic leaves in the semiclassical case. This construction yields a Dixmier map from the space of symplectic leaves in $\CC[G]$ to the space primitive ideals in $R_q[G]$. By construction this map agrees with Yakimov's Dixmier map in \cite[Sec 4]{Yak-spec}, however, it 
 contains  more information regarding the topology of $spec(R_q[G])$ and $P.spec(\CC[G])$, and thus provides new evidence for the following cnjecture.
 
  \begin{conjecture} \cite[Conjecture 4.7]{Yak-spec}
  \label{conj:homeo -spec} 
 The Dixmier map constrcuted above is a homeomorphism between  space of symplectic leaves in $\CC[G]$ and the space primitive ideals in $R_q[G]$.
 \end{conjecture}

   In particular, it is well-known that $G^{u,v}$ lies in the closure of $G^{u',v'}$ if and only if $u\le u'$ and $v\le v'$ in the (strong) Bruhat order, that is there exist  reduced expressions ${\bf u'}=s_{i_1}\ldots s_{i_r}$ and ${\bf u'}=s_{j_1}\ldots s_{j_s}$ such that  there exist $g\le r$ and $h\le s$ with the property that $s_{i_1}\ldots s_{i_g}$ and $s_{j_1}\ldots s_{j_h}$ are reduced expressions for $u$, resp.~$v$. An analogous result is true for the closed quantum double Bruhat  cells. This implies that the intersections of $\CC[G]$ and $R_q[G]$ with the (skew) polynomial rings generated by the clusters ${\bf x}_{{\bf w}_0, {\bf w}_0'}$  (see  the previous section) capture all inclusion relations between Poisson primitive ideals in the classical case and primitive ideals in the quantum case. Of course, it is possible that they yield too many inclusions. It would therefore be of the great interest to determine the exact inclusion relations between symplectic leaves in $\CC[G]$. We therefore make the following conjecture which would immediately imply that the inverse of the Dixmier map is inclusion preserving.
   
   \begin{conjecture}
   \label{conj:incl}
   Let ${\bf x}_{{\bf w}_0, {\bf w}_0'}$ be a cluster as defined above , and  let $S_{\bf x}=\CC[x_1,\ldots, x_{2 k+r}]$ be the subring  of $\CC[G]$ generated by the cluster variables. If $\II$ and $\JJ$ are two Poisson primitive ideals, then $\II\subset  \JJ$ if and only if $(\II\cap S_{\bf x})\subset (\JJ\cap S_{\bf x})$.
   \end{conjecture}  
 
 \begin{remark}
 Notice that it follows from \cite[Lemma 9.4]{Goo1} that proving Conjecture \ref{conj:incl} and its converse would yield a proof  of Conjecture \ref{conj:homeo -spec}.  
 \end{remark}
 
\subsection{Open Questions and Conjectures}

So far we have only considered clusters consisting entirely of minors. However, it is natural to consider other clusters as well and this would allow us to compare inclusion relations between arbitrary strata of the spectrum. We therefore propose the following conjecture.

\begin{conjecture}
Let $\II$ and $\JJ$ be torus invariant proper   prime ideals of $R_q[G]$ with $\II\subset \JJ$. Then there exists a supertoric cluster ${\bf x}$ with COS of depth $r$ such that (up to a suitable permutation of indices) there exist $1\le i\le j\le r$ with $\II=\II_j$ and $\JJ=\II_i$ in the respective stratification $\II_{r}\subset \II_n\ldots\subset  \II_0=R_q[G]$.  
\end{conjecture}

\begin{conjecture}
Every cluster is supertoric.
\end{conjecture}
And finally, we conjecture that the cluster algebra structures give rise to the  Dixmier map intropduced above.
\begin{conjecture}
Let $\AA_q$ be the cluster algebra structure on $R_q[G]$ and let ${\bf x}$ be any cluster. Then any COS stratification gives rise to a Dixmier-type map, and it is a restriction of the map defined in Proposition \ref{pr:Dixmier type map}.
\end{conjecture}

\begin{appendix}
 
  \section{Torus Invariant prime ideals in non-commutative Noetherian rings}
 \label{se: goodearl letzter strat}
 In this section we recall some basic facts from \cite[Section
II]{brown-goodearl}.  Let $R$ be a ring. A two-sided ideal $\II\subset R$ is called
a {\it prime ideal} if   $arb\in \II$ for all $r\in R$,   implies that $a\in \II$ or
$b\in \II$. A two-sided ideal $\II\in R$ is called {\it primitive} if it is the
maximal two-sided ideal contained in a maximal left ideal. We denote the set of
prime ideals by $spec(R)$ and the set of primitive ideals by $prim(R)$. Both
sets come with  a natural Zariski-type topology, where the closed sets are
subsets  of $spec(R)$, resp.~$prim(R)$ such that  $V(\II)=\{P\in spec(R):
P\supset \II\}$, resp.~$V(\II)=\{P\in prim (R): P\supset \II\}$.
 
 Now let $H$ be a group acting by automorphisms on $R$. $H$ permutes the prime,
primitive and left-maximal ideals of $R$. An ideal is called $H$-stable if
$H(\II)=\II$ and, following the notation of \cite[Section II.1]{brown-goodearl}
we denote by $(\II:H)$ the largest $H$-stable  ideal contained in $H$, i.e.~
 $$(\II:H)=\bigcap_{h\in H} h(\II)\ .$$
 
 We can now define the  key notion of an $H$-prime ideal. The ring $R$ is called
$H$-prime, if $R$ is nonzero and any product of nonzero $H$-ideals is nonzero.
 \begin{definition}
 A proper  $H$-ideal $\II$ is called an $H$-prime ideal if $R/\II$ is an
$H$-prime ring.
 \end{definition} 
 If $\II$ is a prime ideal, then it is easy to see that $(\II:H)$ is  $H$-prime.
Note that in general a $H$-prime ideal need not be prime (see \cite[Section
II.1.9]{brown-goodearl}), however in the case we are most interested in, we have
the following fact.

\begin{lemma}\cite[Corollary II.1.12]{brown-goodearl}
\label{le:H-primes are primes}
Let $R$ be a Noetherian ring and let $k$ be an algebraically closed field, and
suppose that  $H$ is a $k$-torus acting  on $R$ by automorphisms   ($R$ is not
necessarily a $k$-algebra).  Then all $H$-prime ideals of $R$ are prime.
\end{lemma}
Denote the set of $H$-prime ideals of $R$  by $H-spec(R)$. 
We naturally obtain a stratification of $spec(R)$ by the $H$-prime ideals
$$ spec(R)=\bigsqcup_{J\in H-spec(R)} spec_J(R)  \ ,$$
where $spec_J(R)=\{P\in spec(R): (P:H)=J\}$.
 Goodearl and Letzter established a general description of the
prime and primitive spectra of a large class of Noetherian algebras which
encompasses many quantized coordinate rings.  In particular, they obtained  the
following results.
 
 \begin{theorem}
 \label{th:stratification 1}
(a) \cite[Theorem II.8.10]{brown-goodearl} Let $A$ be a Noetherian algebra over
the complex numbers,   and let $H$ be a complex affine algebraic group, acting
rationally on $A$  by algebra automorphisms.   Assume that $H-spec(A)$ is
finite. Then, the primitive ideals are exactly the ideals which are maximal in
their $H$-strata. 

(b) \cite[Theorem II.8.14]{brown-goodearl} Moreover, the $H$-orbits within
$prim(A)$ coincide with the $H$-strata of $prim(A)$. In particular there are
only finitely many $H$-orbits in $prim(A)$.  
\end{theorem} 

Moreover, we  can classify the primitive
ideals in each stratum.
  
\begin{theorem}
\cite[Theorem II.2.13]{brown-goodearl} 
\label{th:stratification 2}
 Let $A$ be a  Noetherian $k$-algebra, $k$ infinite, and let $H$ be an algebraic
torus acting rationally on $A$ by $k$-algebra automorphisms, and let $J\in
H-spec(A)$. 
\begin{enumerate}
\item $J$ is a prime ideal.
\item Let $\mathfrak{E}_J$ denote the set of all regular $H$-eigenvectors in
$A/J$. Then $\mathfrak{E}_J$ is a denominator set, and the localization
$A_J=(A/J)[\mathfrak{E}_J^{-1}]$ is an $H$-simple ring with respect to the
induced $H$-action.

\item $spec_J(A)$ is homeomorphic to $spec(A_J)$ via localization and
contraction.
\item $spec(A_J)$ is homeomorphic to $spec(Z(A_J))$ via contraction and
extension, where $Z(A_J)$ denotes the center of $A_J$.
\item  $Z(A_J)$ is a Laurent polynomial ring. The indeterminates can be chosen
to be $H$-eigenvectors with linearly independent $H$-eigenvalues.
\end{enumerate}
\end{theorem}

\section{Poisson and Quantum Planes and Tori}
\label{se:Poisson and quantum tori}
In this section we introduce  quantum affine spaces and tori, as well as their Poisson analogues, over a field $k$ of characteristic zero. Denote by $spec(A)$ the spectrum of prime ideals in an algebra $A$ and by $P.spec(A)$ the spectrum of Poisson prime ideals in a Poisson algebra $(A,\{\cdot,\cdot\})$--these are ideals which are prime and  stable under the Poisson bracket. Both spaces have a natural Zariski-type topology (see e.g. \cite{Goo1}) which means that $I$ lies in the closure of $J$ if and only if $J\subset I$.

\begin{definition}
Let $\Lambda$ be a skew-symmetric $n\times n$-matrix. 

\noindent(a) The Poisson plane $(k[x_1,\ldots,x_n],\Lambda)$ is the Poisson algebra  $k[x_1,\ldots,x_n]$ with bracket defined by
\begin{equation}\label{eq:plane rel}
\{x_i,x_j\}=\lambda_{ij}x_ix_j\ .\end{equation}
 
 \noindent(b) The Poisson  torus $(k[x_1^{\pm 1},\ldots,x_n^{\pm1}],\Lambda)$ is the Poisson algebra defined on $k[x_1^{\pm 1},\ldots,x_n^{\pm 1}]$ with bracket defined by \eqref{eq:plane rel}.
 
 \noindent(c) The quantum plane $k_\Lambda[x_1,\ldots,x_n]$ is the skew-symmetric algebra  defined by the commutation relations
 \begin{equation}\label{eq:q-plane rel}
 x_ix_j=q^{\lambda_{ij}} x_jx_i\ .\end{equation}

\noindent(c) The quantum torus $k_\Lambda[x_1^{\pm 1},\ldots,x_n^{\pm 1}]$ is the skew-symmetric algebra  defined by the commutation relations \eqref{eq:q-plane rel}.
\end{definition}

Let   $c_1,\ldots c_t\in \ZZ^n$ be a basis for the kernel of 
$\Lambda$. Recall that the Poisson center of a Poisson algebra $A$ is the subalgebra $C\subset A$ containing all elements $c\in A$ such that $\{c,a\}=0$ for all $a\in A$. It is easy to observe that the Poisson center of the Poisson torus $(k[x_1^{\pm 1},\ldots,x_n^{\pm1}],\Lambda)$ is the torus $k[x^{\pm c_1},\ldots x^{\pm c_n}]$ where, as usual we write $x^{c_i}=x_1^{(c_i)_1}\ldots x_n^{(c_i)_n}$. Equally well-known is the fact that the center of the quantum torus  $k_\Lambda[x_1^{\pm 1},\ldots,x_n^{\pm 1}]$  is $k[x^{\pm c_1},\ldots, x^{\pm c_n}]$. We have the following fact (see e.g. \cite{Goo1}).

\begin{lemma}
(a) The Poisson spectrum of  the Poisson torus $(k[x_1^{\pm 1},\ldots,x_n^{\pm1}],\Lambda)$ is homeomorphic to the spectrum of $k[x^{\pm c_1},\ldots ,x^{\pm c_n}]$. All Poisson prime ideals are induced by the prime ideals in $k[x^{\pm c_1},\ldots, x^{\pm c_n}]$.

\noindent(b)  The   spectrum of  the quantum torus $k_\Lambda[x_1^{\pm 1},\ldots,x_n^{\pm1}]$ is homeomorphic to the spectrum of $k[x^{\pm c_1},\ldots , x^{\pm c_n}]$. All Poisson prime ideals are induced by the prime ideals in $k[x^{\pm c_1},\ldots, x^{\pm c_n}]$.
\end{lemma}

We can now describe the spectra of the quantum and Poisson planes. Let ${\bf i}\subset [1,n]$. Denote by $\Lambda_{\bf i}$ the submatrix of $\Lambda$ consisting of the rows and columns labeled by ${\bf i}=\{i_1<\ldots <i_k\}$. Since the sets $\{x_1,\ldots, x_n\}-\{x_{i_1},\ldots, x_{i_k}\}$ generate the torus invariant Poisson prime ideals of $(k[x_1,\ldots,x_n],\Lambda)$, resp.~the torus invariant prime ideals in $k_\Lambda[x_1,\ldots,x_n]$, we obtain the following fact from the results of Section \ref{se: goodearl letzter strat}.
\begin{lemma}
(a) Let $\Lambda$ be a skew-symmetric $n\times n$-matrix. For each ${\bf i}\subset [1,n]$, there exist injective continuous maps $$spec(k_{\Lambda_{\bf i}}[x_{i_1}^{\pm 1},\ldots,x_{i_k}^{\pm 1}])\hookrightarrow spec(k_\Lambda[x_1,\ldots,x_n])\ ,$$ where we set $k_{\Lambda_{\bf i}}[x_{i_1}^{\pm 1},\ldots,x_{i_k}^{\pm 1}]=k$ if ${\bf i}=\emptyset$. Analogously,$$ P.spec(k[x_{i_1}^{\pm 1},\ldots,x_{i_k}^{\pm1}],\Lambda_{i})\hookrightarrow P.spec((k[x_1^{\pm 1},\ldots,x_n^{\pm1}],\Lambda)\ .$$

(b) Moreover, 
$$spec(k_\Lambda[x_1,\ldots,x_n])=\bigsqcup\limits_{{\bf i}\subset [1,n]} spec(k_{\Lambda_{\bf i}}[x_{i_1}^{\pm 1},\ldots,x_{i_k}^{\pm 1}])\ ,$$
$$ P.spec((k[x_1^{\pm 1},\ldots,x_n^{\pm1}],\Lambda)=\bigsqcup\limits_{{\bf i}\subset [1,n]}P.spec(k[x_{i_1}^{\pm 1},\ldots,x_{i_k}^{\pm1}],\Lambda_{i})\ .$$
\end{lemma}

Finally, the following fact allows us to compare inclusion relations between strata in the quantum and Poisson cluster algebras.

\begin{lemma}
\label{le:space homeo}
Let $\Lambda$ be a skew-symmetric $n\times n$-matrix. Then the identity map on $[1,n]$ sending $x_i\in k_\Lambda[x_1,\ldots,x_n]$ to $x_i\in(k[x_1^{\pm 1},\ldots,x_n^{\pm1}],\Lambda)$ induces a homeomorphism from $spec(k_\Lambda[x_1,\ldots,x_n])\to  P.spec((k[x_1^{\pm 1},\ldots,x_n^{\pm1}],\Lambda)$.
\end{lemma}
\end{appendix}

\end{document}